\documentclass[12pt]{amsart}

\author{Paul \textsc{Poncet}}
\address{CMAP, \'{E}cole Polytechnique, Route de Saclay, 91128 Palaiseau Cedex, France \\
and INRIA, Saclay--\^{I}le-de-France}

\email{poncet@cmap.polytechnique.fr}

\usepackage{stmaryrd}
\usepackage{amssymb}
\usepackage{amsmath}
\usepackage{graphicx}
\usepackage{t1enc}

\usepackage{yhmath}
\usepackage{calrsfs} 
\usepackage{times} 
\def\twoheaduparrow{\rlap{$\uparrow$}\raise.5ex\hbox{$\uparrow$}}

\newcommand{\argmax}{\operatornamewithlimits{argmax}}
\newcommand{\argmin}{\operatornamewithlimits{argmin}}

\DeclareMathOperator*{\Min}{Min}
\DeclareMathOperator*{\co}{co}
\DeclareMathOperator*{\ex}{ex}

\DeclareMathOperator*{\nsq}{NS_4}

\newtheorem{theorem}{Theorem}[section]
\newtheorem{corollary}[theorem]{Corollary}
\newtheorem{proposition}[theorem]{Proposition}
\newtheorem{lemma}[theorem]{Lemma}

\theoremstyle{definition}

\newtheorem{example}[theorem]{Example}
\newtheorem{remark}[theorem]{Remark}
\newtheorem{problem}[theorem]{Problem}
\newenvironment{acknowledgements}[1][]{\par\vspace{0.5cm}\noindent\textbf{Acknowledgements#1.} }{\par}

\begin{document}

\title{Convexities on ordered structures \\ have their Krein--Milman theorem} 

\date{\today}



\begin{abstract}
We show analogues of the classical Krein--Milman theorem for several ordered algebraic structures, especially in a semilattice (non-linear) framework. 
In that case, subsemilattices are seen as convex subsets, and for our proofs we use arguments from continuous lattice theory and abstract convexity theory. 
\end{abstract}

\maketitle

\section{Introduction} 

A semilattice is a commutative semigroup $(S, \oplus)$ in which all elements $t$ are idempotent, i.e.\ such that $t \oplus t = t$. Such an $S$ is endowed with a natural partial order defined by $s \leqslant t \Leftrightarrow s \oplus t = t$, so that $s \oplus t$ is the supremum of the pair $\{s, t \}$. 
Semilattices have been widely explored in the last decades; a key result of the theory is the ``fundamental theorem of compact semilattices'', which identifies the category of complete continuous semilattices with that of compact topological semilattices with small semilattices. The statement is due to Hofmann and Stralka \cite{Hofmann76}. Lawson's contribution was decisive for its discovery (see \cite{Lawson69}, \cite{Lawson73}). See also Lea \cite{Lea76} for an alternative proof and Gierz et al.\ \cite[Theorem~VI-3.4]{Gierz03}. 
This theorem draws a link between the \textit{algebraic} and the \textit{topological} natures of semilattices. 

But semilattices can also be regarded as \textit{geometric} objects, where subsemilattices are treated as convex subsets. Surprisingly, this point of view has been hardly considered in the literature. 
Exceptions are the work of Jamison (\cite{Jamison74}, \cite{Jamison77}, \cite[Appendix]{Jamison81}, \cite{Jamison82}) cited by van de Vel (\cite{vanDeVel85}, \cite{vanDeVel93}, \cite{vanDeVel93b}), and a comment by Gierz et al. \cite[p.\ 403]{Gierz03}. 

One reason is certainly that a semilattice with a least element can be seen as a module over the idempotent semifield $\mathbb{B} = \{0, 1\}$. Therefore, it belongs to the more general class of modules over an idempotent semifield $(\Bbbk, \oplus, \times)$ (see \cite{Poncet11}), and it happens that these structures have been deeply studied in the framework of ``max-plus'' or ``tropical'' convexity. We have addressed these aspects in \cite[Chapter~V]{Poncet11}. 

However, semilattices should not be reduced to a special case of modules over an idempotent semifield. Indeed, the use of the set $\mathbb{B}$ as a finite idempotent semifield creates unusual phenomena: for instance, a $\mathbb{B}$-module of finite type is finite; a convex subset of a $\mathbb{B}$-module is not connected in general. 
So one should expect $\mathbb{B}$-modules to have discontinuous behaviour, that one does not usually observe in modules such as $\mathbb{R}_+^n$ (over the idempotent semifield $\mathbb{R}^{\max}_+ = (\mathbb{R}_+, \max, \times)$). 

It is also worth studying semilattices before modules, because if $K$ is a convex subset of an $\mathbb{R}^{\max}_+$-module $M$, then the set $\{ (r . x, r) : x \in K, r \in [0,1] \}$ is a subsemilattice of the semilattice $M \times [0, 1]$. This partly explains why results on semilattices shall be useful for applications to the geometry of $\mathbb{R}^{\max}_+$-modules. 

Other convexities naturally arise on ordered structures such as partially ordered sets, semilattices and lattices. For instance, a semilattice can also be endowed with the convexity made up of its \textit{order-convex} subsemilattices; this case was notably studied by Jamison \cite{Jamison77} and van de Vel (\cite{vanDeVel85}, \cite{vanDeVel93b}, \cite{vanDeVel93}). See also Horvath and Llinares Ciscar \cite{Horvath96}  
and Nguyen The Vinh \cite{NguyenTheVinh05} for investigations on path-connected topological semilattices. 

For this series of convexity structures, we prove analogues of classical results of convex analysis such as the Krein--Milman theorem (Krein and Milman \cite{Krein40}, see also Bourbaki \cite{Bourbaki81}) and Milman's converse \cite{Milman47}. 
For semilattices equipped with the convexity of subsemilattices, our main result is the following: 

\begin{theorem}\label{thm:klee0}
Let $S$ be a locally convex topological semilattice. Then every locally compact, weakly-closed, convex subset of $S$ containing no line is the weakly-closed convex hull of its extreme points. 
\end{theorem}

Local convexity here is another way to say that $S$ has \textit{small semilattices} in the sense of Lawson \cite{Lawson69}. The concept of \textit{line}, which is intuitive in classical analysis, needs to be properly defined in this non-linear context. The \textit{weak topology} refers to the topology generated by the family of continuous semilattice-morphisms from $S$ to $[0,1]$. Because of the fundamental theorem of compact semilattices, our proofs directly or indirecly use methods and elements from domain theory. 

Numerous Krein--Milman theorems have been proved in the literature. Yet they do not enable one to deduce directly the above theorem. 
For instance Fan \cite[Lemma~3]{Fan63} gave a set-theoretic definition of an \textit{extremality}, a concept close to the notion of \textit{face} in classical analysis;  
then he used it to prove an abstract Krein--Milman type theorem (see also \cite[Theorem~IV-2.6]{vanDeVel93}). 
However, his definition and adds-on by others such as Lassak \cite{Lassak86} remain driven by classical convexity theory, where addition is a \textit{cancellative} binary relation; it does not work in an idempotent setting. 

Another result of this kind  
is due to Wieczorek \cite{Wieczorek89}. It requires two conditions: that every singleton be convex, and that the family of upper-semicontinuous \textit{strictly convex} real-valued maps separate convex closed subsets and points (see also \cite[Topic~IV-2.30]{vanDeVel93}). However, for ordered structures, the former condition may not be satisfied (we shall see examples such as the upper convexity on a poset, or the ideal convexity on a semilattice), and the latter seems too complex for practical verification. 

We also examine the case of topological semilattices with finite \textit{breadth} $b$, that happen to be always locally convex. Jamison \cite[Paragraph~4.D]{Jamison82} remarked that breadth coincides with the Carath\'{e}odory number associated with the convexity of subsemilattices. We prove a Minkowski type theorem, which asserts that under appropriate hypothesis every point is the join of at most $b$ extreme points. The \textit{depth} of the semilattice also coincides with an interesting convexity invariant, namely the Helly number, and we establish links between depth and the number of extreme points of a compact convex subset. 

The paper is organized as follows. 
Section~\ref{sec:convth} gives basics of abstract convexity theory. 
In Section~\ref{sec:or} we recall Wallace's lemma on the existence of minimal elements in compact partially ordered sets, which will reveal its importance for the existence of extreme points in compact ordered structures. We also show a Krein--Milman type theorem in partially ordered sets. 
Section~\ref{sec:convsemilat} introduces the main convexity examined in this work, which is the convexity made up of the subsemilattices of a semilattice. We prove that a Krein--Milman type theorem also holds, and see that it essentially comes from the result that coirreducible elements are order-generating in continuous semilattices. An analogous form of Bauer's principle is also proved. 
Section~\ref{sec:convsemilat2} goes one step further: after the work of Klee in classical convex analysis, we prove that the Krein--Milman theorem holds for locally compact weakly-closed convex subsets containing no line, with an adequate definition of line in topological semilattices. Also, Milman's converse is proved. 
Topological semilattices with finite breadth or with finite depth are considered in Section~\ref{sec:fb}. We recall that the breadth and the Carath\'eodory number of a semilattice coincide, and we prove a Minkowski type theorem. 
Other convexities on semilattices and lattices are proposed in Section~\ref{sec:cg}. We provide necessary and sufficient conditions for these convexities to be convex geometries, which is a minimal requirement for Krein--Milman type theorems.

\section{Reminders of abstract convexity}\label{sec:convth}

A collection $\mathrsfs{C}$ of subsets of a set $X$ is a \textit{convexity} (or an \textit{alignment}) on $X$ if it satisfies the following axioms: 
\begin{enumerate}
	\item[-] $\emptyset, X \in \mathrsfs{C}$, 
	\item[-] $\mathrsfs{C}$ is closed under arbitrary intersections, 
	\item[-] $\mathrsfs{C}$ is closed under directed unions. 
\end{enumerate}
The last condition means the following: if $\mathrsfs{D} \subset \mathrsfs{C}$ is such that, for all $C_1, C_2 \in \mathrsfs{D}$, there is some $C \in \mathrsfs{D}$ containing both $C_1$ and $C_2$, then $\bigcup \mathrsfs{D} \in \mathrsfs{C}$. 
The pair $(X,\mathrsfs{C})$ is then a \textit{convexity space}. 
Elements of $\mathrsfs{C}$ are called \textit{convex} subsets of $X$. 
If $A \subset X$, the \textit{convex hull} $\co(A)$ of $A$ is the intersection of all convex subsets containing $A$. 
Convex subsets that are the convex hull of some finite subset are called \textit{polytopes}. They are of special importance for they generate the whole convexity, in the sense that $C \subset X$ is convex if and only if, for 
every finite subset $F$ of $C$, $\co(F) \subset C$. 

The wording of the Krein--Milman theorem includes the notion of \textit{extreme point} of a subset $A \subset X$, which is an element $x$ of $A$ such that $x \notin \co(A \setminus \{ x \})$, or equivalently, if $A$ is convex, such that $A \setminus \{x\}$ is convex. 
The set of extreme points of $A$ is denoted by $\ex A$. 

In practice, $X$ will be a convexity space endowed with a \textit{compatible} topology, that is a topology making every polytope (topologically) closed. Then $X$ will be called a \textit{topological convexity space}. 

For more background on abstract convexity, see the monograph of van de Vel \cite{vanDeVel93}. 
Other attempts and approaches, that we do not consider here, have been made by mathematicians to generalize the concept of convexity; see for instance Singer \cite{Singer97} or Park \cite{Park09}.

\section{Convexities on partially ordered sets}\label{sec:or}

In this section we recall (and discuss) Wallace's lemma (see Wallace \cite[Paragraph~2]{Wallace45}), that we shall use several times later on, and we interpret it as a Krein--Milman type theorem for partially ordered sets. We also prove a converse statement, known as Milman's converse in the framework of locally convex topological vector spaces.

\subsection{Wallace's lemma}

A \textit{partially ordered set} or \textit{poset} $(P,\leqslant)$ is a set $P$ equipped with a reflexive, transitive, and antisymmetric binary relation $\leqslant$. If $A \subset P$, we denote by $\downarrow\!\! A$ the \textit{lower} subset generated by $A$, i.e.\ $\downarrow\!\! A := \{ x \in P : \exists a \in A, x \leqslant a\}$, and we write $\downarrow\!\! x$ for the \textit{principal ideal} $\downarrow\!\!\{x\}$. \textit{Upper} subsets $\uparrow\!\! A$ and \textit{principal filters} $\uparrow\!\! x$ are defined dually. 
A topology on a poset is \textit{lower semiclosed} (resp.\ \textit{upper semiclosed}) if each principal ideal (resp.\ principal filter) is a closed subset. It is \textit{semiclosed} if it is both lower semiclosed and upper semiclosed. 

Note that our definition of a \textit{compact} subset of a topological space does not assume Hausdorffness. 

\begin{proposition}[Wallace's lemma, \protect{\cite[Paragraph~2]{Wallace45}}]
Let a poset be equip\-ped with a lower semiclosed topology. Then every nonempty compact subset has a minimal element. 
\end{proposition}

We take advantage of this reminder to stress that we found no explicit statement in the literature of the following equivalence. 
Recall first that the Ultrafilter Principle (alias the Prime Ideal Theorem), which says that every filter on a set is contained in an ultrafilter, is strictly weaker than the axiom of choice. 

\begin{proposition}
Wallace's lemma for all posets together with the Ultrafilter Principle are equivalent to the axiom of choice. 
\end{proposition}

\begin{proof}
Necessity is made clear by the proof of \cite[Proposition~VI-5.3]{Gierz03}, which makes use of Hausdorff's maximality principle to prove Wallace's lemma. For sufficiency, let $P$ be a poset, and let $L$ be a linearly ordered subset (or \textit{chain}) of $P$. Let $\mathrsfs{L}$ be the (nonempty) collection of chains of $P$ containing $L$, ordered by reverse inclusion. Then $\mathrsfs{L}$ is a complete semilattice (i.e.\ a semilattice in which every nonempty subset has a supremum and every filtered subset has an infimum, see Section~\ref{sec:convsemilat}), hence is compact when equipped with the Lawson topology (see \cite[Theorem~III-1.9]{Gierz03}; its proof uses Alexander's lemma, which itself is known to be implied by the Ultrafilter Principle). By Wallace's lemma, $\mathrsfs{L}$ has a minimal element, i.e.\ there is a maximal chain in $P$ containing $L$. This proves Hausdorff's maximality principle, which is equivalent to the axiom of choice. 
\end{proof}

\subsection{Krein--Milman theorems for posets}

Actually, the result \cite[Proposition~VI-5.3]{Gierz03}, used in the previous proof, refines Wallace's lemma: under the same hypothesis, it concludes that, if $K$ is a compact subset and $x \in K$, there is some minimal element of $K$ below $x$. We interpret this version as a Krein--Milman type theorem for partially ordered sets endowed with the \textit{upper convexity} made up of upper subsets (see Edelman and Jamison \cite[Theorem~3.2]{Edelman85} for a characterization of this convexity). 
In this setting, extreme points of a convex subset $K$ coincide with its minimal elements $\Min K$, and if $K$ is compact convex, then $K = \co(\ex K) = \uparrow\!\! (\Min K)$ (see Figure~\ref{fig:Schema1Chapitre4}). Note the absence of topological closure in this equality. 

\begin{figure}
	\centering
		\includegraphics[width=0.7\textwidth]{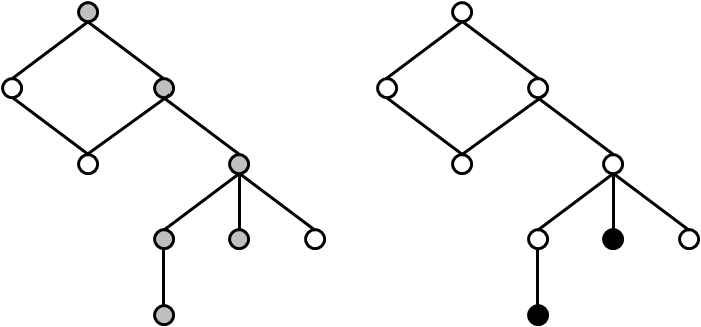}
	\caption{Hasse diagram of a finite partially ordered set (with the discrete topology). The gray points (on the left) define a convex subset (with respect to the upper convexity); the black points (on the right) are its minimal elements. }
	\label{fig:Schema1Chapitre4}
\end{figure}

\begin{theorem}[Krein--Milman for posets I]\label{thm:kmp}
Consider a poset with the upper (resp.\ lower) convexity, and equipped with a lower semiclosed (resp.\ an upper semiclosed) topology. Then every compact subset $K$ satisfies 
\[
\co(K) = \co(\ex K). 
\]
\end{theorem}

\begin{proof}
A direct consequence (and actually, an equivalent form) of Wallace's lemma is the following: if $K$ is a compact subset and $x \in K$, there is some minimal element of $K$ below $x$. To see this, it suffices to apply Wallace's lemma to the nonempty compact subset $K \cap \downarrow\!\! x$. 
Then $K \subset \uparrow\!\! (\Min K) = \co(\ex K)$, so that $\co(K) = \uparrow\!\! K = \co(\ex K)$. 
\end{proof}

Franklin \cite{Franklin62}, Baker \cite{Baker69} or Jamison \cite{Jamison81} preferentially considered po\-sets endowed with their \textit{order convexity}. This convexity, introduced by Birkhoff \cite{Birkhoff48}, is  generated by intervals $[x, y] = \uparrow\!\! x \cap \downarrow\!\! y = \{ z : x \leqslant z \leqslant y\}$. See Jamison \cite{Jamison79} for various characterizations of order convexity. See also Birkhoff and Bennett \cite{Birkhoff85}. 
Here, convex subsets are subsets of the form $\uparrow\!\! A \cap \downarrow\!\! A$, extreme points are the elements $e$ such that $e \in [x, y] \Rightarrow e \in \{x, y\}$, i.e.\ are either minimal elements or maximal elements, and Franklin \cite[Theorem~III]{Franklin62} and Baker \cite[Theorem~1]{Baker69} proved that a Krein--Milman type theorem also holds.  

\begin{theorem}[Krein--Milman for posets II, \protect{\cite[Theorem~III]{Franklin62}} and \protect{\cite[Theorem~1]{Baker69}}]\label{thm:kmpo}
Consider a poset with the order convexity, and equipped with a semiclosed topology. Then every compact subset $K$ satisfies 
\[
\co(K) = \co(\ex K). 
\]
\end{theorem}

See also Wirth \cite[Theorem~1]{Wirth74} for a Krein--Milman type theorem in certain posets equipped with the open-interval topology. 

It is remarkable that, in Theorems~\ref{thm:kmp} and \ref{thm:kmpo}, the Krein--Milman property holds without any local convexity hypothesis. Local convexity is certainly automatic in every compact \textit{pospace} (defined as a poset $P$ equipped with a topology making the partial order closed in $P \times P$), as is well known since the work of Nachbin \cite{Nachbin65}, but Theorems~\ref{thm:kmp} and \ref{thm:kmpo} do not need this assumption.

\subsection{Milman's converse}

In classical convex analysis, Milman's theorem \cite{Milman47} is probably as important as the Krein--Milman theorem itself, for it asserts that the representation of a compact convex subset as the closed convex hull of its extreme points is, in some sense, optimal. That is, for every such representation, the ``representing'' subset, if closed, contains the subset of extreme points. 
Fortunately, a similar result holds in pospaces. For the next assertion, we write $\overline{A}$ for the topological closure of a subset $A$.

\begin{theorem}[Milman for posets]
Let $P$ be a pospace with the upper (resp.\ lower, order) convexity, and $K$ be a closed convex subset of $P$. Then, for every compact subset $A$ of $K$ such that $K =  \overline{\co}(A)$, we have $A \supset \ex K$. 
\end{theorem}

\begin{proof}
We consider the case of upper convexity only. 
Since $P$ is a pospace and $A$ is compact in $P$, $\co(A) = \uparrow\!\! A$ is closed in $P$ by \cite[Proposition~VI-1.6(ii)]{Gierz03}, hence $K = \uparrow\!\! A$. Thus, $\ex K = \Min K = \Min (\uparrow\!\! A) \subset A$. 
\end{proof}

\section{The algebraic convexity of a semilattice}\label{sec:convsemilat}

\subsection{Introduction}

A \textit{semilattice} $S$ is a poset in which every nonempty finite subset $F$ has a supremum, denoted by $\bigoplus_S F$ (or by $\bigoplus F$ when the context is clear). If $x, y \in S$, we write $x \oplus y$ for $\bigoplus \{x, y\}$. 

We endow the semilattice $S$ with its \textit{algebraic convexity} made up of its subsemilattices, i.e.\ the subsets $T$ of $S$ such that $x \oplus y \in T$ whenever $x, y\in T$ (in particular the empty set is a subsemilattice). 
We shall also say that subsemilattices are \textit{convex} subsets of $S$. 
If $A \subset S$, the \textit{convex hull} $\co(A)$ of $A$ is the subsemilattice generated by $A$ (see Figure~\ref{fig:Schema2Chapitre4}). 

\begin{figure}
	\centering
		\includegraphics[width=0.7\textwidth]{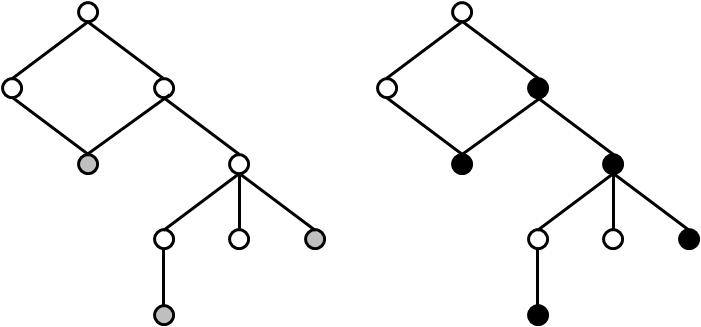}
	\caption{Hasse diagram of a finite semilattice. The gray points (on the left) define a subset; the black points (on the right) are its convex hull (with respect to the algebraic convexity), which here is not connected. }
	\label{fig:Schema2Chapitre4}
\end{figure}

The algebraic convexity of a semilattice deserves special attention, for it has been hardly considered in the literature. Exceptions are the work of Jamison (\cite{Jamison74}, \cite{Jamison77}, \cite[Appendix]{Jamison81}, \cite{Jamison82}) and a comment by Gierz et al. \cite[p.\ 403]{Gierz03}. 
However, recall from the Introduction that a semilattice with a bottom element is equivalently described as a $\mathbb{B}$-module with $\mathbb{B} = \{0, 1\}$, thus is a special case of module over an idempotent semifield. Consequently, the algebraic convexity of a semilattice with a bottom element is the same as the \textit{tropical convexity} of the associated $\mathbb{B}$-module. Tropical convexity has been the subject of a great amount of research, and we refer the reader to \cite[Chapter~V]{Poncet11} for background and references. 

It should be stressed that other interesting convexities can be defined on semilattices, for instance the \textit{ideal convexity} 
consisting of lower subsemilattices, or the \textit{order-algebraic convexity} made up of \textit{order-convex} subsemilattices, that is subsemilattices $T$ such that $x \leqslant y \leqslant z$ and $x, z \in T$ imply $y \in T$. Information on the latter convexity may be gathered from Jamison \cite{Jamison77} and van de Vel (\cite{vanDeVel85}, \cite{vanDeVel93b}, \cite{vanDeVel93}), and we shall discuss several convexities in more detail in Section~\ref{sec:cg}. 

If $K$ is a subset of the semilattice $S$, then $x \in K$ is an extreme point of $K$ if and only if $x$ is \textit{coirreducible} in $K$, i.e., for every nonempty finite subset $F$ of $K$, $x = \bigoplus F \Rightarrow x \in F$ (see Figure~\ref{fig:Schema3Chapitre4}). 

The semilattice $S$ is \textit{topological} if it is endowed with a Hausdorff topology such that $S \times S \ni (x, y) \mapsto x \oplus y \in S$ is continuous (where $S \times S$ is equipped with the product topology). Be careful that, in \cite{Gierz03}, a topological semilattice is not supposed Hausdorff, although this hypothesis is made in all other references cited in this work. 
A topological semilattice $S$ can then be seen as a topological convexity space, in which the topological closure of every convex subset remains convex (this is what van de Vel called \textit{closure stability} \cite[Definition~III-1.7]{vanDeVel93}). This latter property can be easily proved using nets. 
Also, $S$ is \textit{locally convex} if every point has a basis of convex neighbourhoods, that is if it has small semilattices in the sense of Lawson \cite{Lawson67} (see also Gierz et al.\ \cite[Definition~VI-3.1]{Gierz03}).

\subsection{Compact local convexity or complete continuity?}

At this stage it is worth recalling the fundamental theorem of compact semilattices (see Hofmann and Stralka \cite[Theorem~2.23]{Hofmann76}, Lea \cite[Theorem]{Lea76}, and Gierz et al.\ \cite[Theorem~VI-3.4]{Gierz03}). 
For this purpose we briefly recall some basic definitions of continuous poset theory. 
A subset $F$ of a poset $(P,\leqslant)$ is \textit{filtered} if it is nonempty and, for all $x, y \in F$, there is a lower bound of $\{x,y\}$ in $F$. 
We say that $y \in P$ is \textit{way-above} $x\in P$, written $y \gg x$, if, for every filtered subset $F$ with an infimum $\bigwedge F$, $x \geqslant \bigwedge F$ implies $y \in \uparrow\!\! F$. 
The poset $P$ is \textit{continuous} if $\twoheaduparrow x := \{ y \in P : y \gg x \}$ is filtered and $x = \bigwedge \twoheaduparrow x$, for all $x \in P$. A \textit{domain} is a continuous poset in which every filtered subset has an infimum. A domain that is also a semilattice is a \textit{continuous semilattice}. 
A semilattice is \textit{complete} if every nonempty subset has a supremum and every filtered subset has an infimum. 

Intervals of (extended) real numbers, with the usual order, for instance $[0,1]$, $[0,1)$, $(0,1)$, are all continuous posets, and the way-above relation coincides with the strict order $>$, except at the top element when it exists (e.g.\ $1 \gg 1$ in $[0,1]$). All these examples are also semilattices, but only $[0,1]$ and $[0,1)$ are domains (thus continuous semilattices), and $[0,1]$ is the only complete semilattice (or complete lattice). 

A subset $A$ of the poset $P$ is \textit{Scott-open} if it is lower and if, whenever $\bigwedge F \in A$ for some filtered subset $F$ of $P$ with infimum, then $F \cap A \neq \emptyset$. The collection of Scott-open subsets of $P$ is a topology, called the \textit{Scott topology}. 
The \textit{Lawson topology} on $P$ is then the topology generated by the Scott topology and the subsets of the form $P \setminus \downarrow\!\! x$, $x \in P$. 

Here comes the announced fundamental theorem of compact semilattices (we skip the identification of morphisms between the two categories at stake). 

\begin{theorem}\cite[Theorem~VI-3.4]{Gierz03}\label{thm:fond}
\begin{enumerate}
	\item Let $K$ be a complete continuous semilattice. Then, with respect to the Lawson topology $K$ is a compact locally convex topological semilattice. 
	\item Conversely, let $K$ be a compact locally convex topological semilattice. Then, with respect to its semilattice structure $K$ is a complete continuous semilattice. Furthermore, the topology of $K$ is the Lawson topology. 
\qed
\end{enumerate}
\end{theorem}

We warn the reader that, considering a locally convex topological semilattice with a complete semilattice structure, the previous theorem cannot be used to assert that $S$ is continuous, nor that the topology is the Lawson topology. 

\begin{problem}
Gierz et al.\ asserted that a (not necessarily complete) continuous semilattice is a \textit{strictly locally convex}  topological semilattice (meaning that every point has a basis of convex open neighbourhoods) for the Lawson topology (see \cite[Exercise~III-2.17]{Gierz03}). Is there any kind of converse statement? 
\end{problem}

\subsection{The Krein--Milman theorem}

With the correspondence given by the fundamental theorem~\ref{thm:fond}, we now prove an analogue of the Krein--Milman theorem for semilattices: 

\begin{theorem}[Krein--Milman for semilattices]\label{thm:km}
Let $S$ be a locally con\-vex topological semilattice. Then every nonempty compact subset of $S$ has at least one extreme point, and every compact convex subset of $S$ is the closed convex hull of its extreme points. 
\end{theorem}

\begin{figure}
	\centering
		\includegraphics[width=0.7\textwidth]{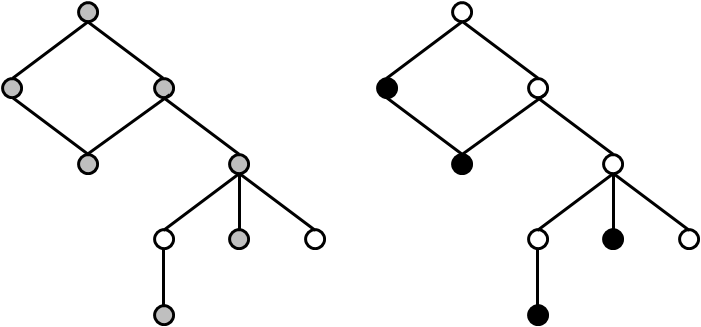}
	\caption{Hasse diagram of a finite semilattice. The gray points (on the left) define a subsemilattice; the black points (on the right) are its coirreducible elements. }
	\label{fig:Schema3Chapitre4}
\end{figure}

\begin{proof}
The former assertion is a direct consequence of Wallace's lemma since every minimal point is extreme.  
The latter comes from an interpretation of \cite[Corollary~I-3.10]{Gierz03}. 
Let $K$ be a nonempty compact convex subset of $S$. By \cite[Proposition~VI-3.2(i)]{Gierz03}, since $S$ is a topological semilattice with small semilattices, $K$ is, in its own right, a compact topological semilattice with small semilattices when equipped with the relative topology. Hence by the fundamental theorem of compact semilattices, $K$ is a complete continuous semilattice. 

Now, a consequence of \cite[Corollary~I-3.10]{Gierz03} is that, in the continuous semilattice $K$, the subset of coirreducible elements (i.e., extreme points) of $K$ is order-generating (see also Hofmann and Lawson \cite[Proposition~2.7]{HofmannLawson76}). This means that every $x$ in $K$ equals $\bigoplus_K (\ex K \cap \downarrow\!\! x)$, where the supremum is taken in $K$. 

To conclude the proof, let $T$ be the topological closure in $S$ of the subsemilattice $\co(\ex K \cap \downarrow\!\! x)$ of $K$.  Since $K$ is closed in $S$, $T$ is also closed in $K$. By the closure stability property (see the Introduction of Section~\ref{sec:convsemilat}), $T$ is then a closed subsemilattice of the compact semilattice $K$.  
 By \cite[Proposition~VI-2.9]{Gierz03}, $T$ is stable by suprema in $K$ of nonempty subsets, hence $x = \bigoplus_K (\ex K \cap \downarrow\!\! x)$ is in $T$.  This proves that $x \in \overline{\co}(\ex K)$, so that $K = \overline{\co}(\ex K)$. 
\end{proof}

\begin{remark}
We can weaken the assumptions of Theorem~\ref{thm:km}, and only suppose that $S$ is a locally convex Hausdorff \textit{semitopological} semilattice, i.e.\ a semilattice equipped with a locally convex Hausdorff topology and a separately continuous addition. 
Indeed, \cite[Theorem~VII-4.8]{Gierz03} then ensures that every compact convex subset of $S$ is still a \textit{topological} semilattice (see also the original paper by Lawson \cite{Lawson74} on semitopological semigroups). 
\end{remark}

The hypothesis of the preceding theorem can be weakened in a different manner.  
We say that a subset $K$ of a semilattice is \textit{principally compact} if $K \cap \downarrow\!\! x$ is compact for all $x \in K$. 

\begin{corollary}
Let $S$ be a locally convex topological semilattice. Then every non\-empty principally compact subset of $S$ has at least one extreme point, and every principally compact closed convex subset of $S$ is the closed convex hull of its extreme points. 
\end{corollary}

\begin{proof}
Let $K$ be a nonempty principally compact subset of $S$, and let $x \in K$. If one notices that $\ex(K \cap \downarrow\!\! x) =  \ex(K) \cap \downarrow\!\! x$, then the first assertion of the corollary is obvious. Now suppose also that $K$ is convex, and let $L = K \cap \downarrow\!\! x$, which is nonempty compact convex. Then, by the Krein--Milman theorem, $x \in L = \overline{\co}(\ex L) = \overline{\co}(\ex(K) \cap \downarrow\!\! x) \subset \overline{\co}(\ex K)$, so that $K = \overline{\co}(\ex K)$. 
\end{proof}

\subsection{Bauer's principle}\label{sec:bauer}

Let $S$ be a topological semilattice and $K$ be a convex subset of $S$, and let $L$ be a chain (considered as a semilattice). 
A map $f : K \rightarrow L$ such that $f(x \oplus y) \leqslant f(x) \oplus f(y)$ (resp.\ $f(x \oplus y) \geqslant f(x) \oplus f(y)$), for all $x, y \in K$, is called \textit{convex} (resp.\ \textit{concave}). An \textit{affine} map is a convex and concave map, i.e.\ a semilattice-morphism. It is easily checked  that $f$ is concave if and only if it is order-preserving. Also, $f$ is convex (resp.\ concave) if and only if its \textit{epigraph} $\{ (x, t) \in K \times L : f(x) \leqslant t \}$ (resp.\ its \textit{hypograph} $\{ (x, t) \in K \times L : f(x) \geqslant t \}$) is convex in $K \times L$. 

We also say that a map $f : K \rightarrow L$ is \textit{lower-semicontinuous} or \textit{lsc} (resp.\ \textit{upper-semicontinuous} or \textit{usc}) if $\{ f > t \}$ (resp.\ $\{ f < t \}$) is open in $K$ for all $t \in L$. 

Let $K$ be a nonempty subset of $S$. A subset $E$ of $K$ is \textit{extreme} in $K$ if, for all $x, y \in K$, $x \oplus y \in E \Rightarrow (x \in E \mbox{ or } y \in E)$, and $E$ is a \textit{face} of $K$ if $E$ is a nonempty compact subset of $K$ that is extreme in $K$. The next result is a semilattice-version of the classical Bauer maximum principle \cite{Bauer60}. 

\begin{proposition}[Bauer's maximum principle]
Let $S$ be a topological semilattice, $K$ be a nonempty compact convex subset of $S$, and $L$ be a chain. Let $f : K  \rightarrow L$ be a convex, usc map. Then $\argmax f$ is a face of $K$, and $f$ attains its maximum on $\ex K$. 
\end{proposition}

\begin{proof}
By compactness of $K$, we classically know that $f$ attains its maximum on $K$. Now let $a = \max_{x \in K} f(x)$, and let $\argmax f$ be the nonempty set $\{ x \in K : f(x) = a\}$.  
The fact that $\argmax f = \{ x \in K : f(x) \geqslant a \}$ and the upper-semicontinuity of $f$ tell us that $\argmax f$ is closed, hence (nonempty) compact. Also, by convexity of $f$ and the fact that $L$ is a chain, $\argmax f$ is extreme in $K$, thus a face of $K$. Hence, every minimal element of $\argmax f$ (which exists by Wallace's lemma) belongs to $\ex K$. 
\end{proof}

\begin{remark}
Lassak \cite{Lassak86} gave, in an abstract convexity setting, a set-theoretic notion of \textit{extreme subset} as follows. For a convexity space $X$ and a subset $K$, he called $E \subset K$ an extreme subset of $K$ if 
\begin{equation}\label{eq:lassak}
E \cap \co(F) \subset \co(E \cap F), 
\end{equation}
for all finite subsets $F$ of $K$. 
However, with this definition, we do not cover the notion of extreme subset introduced above for semilattices. 
Even if Lassak's approach is appropriate for generalizing convexity of vector spaces, it does not fit with the setting of ordered structures that we want to study. 
The following modification in the definition actually untangles this problem, i.e.\ is adequate for both classical and ``ordered'' applications: one should replace (\ref{eq:lassak}) by 
\begin{equation*}
E \cap \co(F) \neq \emptyset \Rightarrow E \cap F \neq \emptyset, 
\end{equation*}
for all finite subsets $F$ of $K$. The transitivity of the relation ``is extreme in'' is then lost, but this is indeed what happens in ordered structures. In particular, in the previous proof, an extreme point of $\argmax f$ would not necessarily give an extreme point of $K$. 
\end{remark}

For completeness, we also give a dual version of Bauer's principle. Here the hypothesis can be weakened. A map $f : K \rightarrow L$ is called \textit{quasiconcave} if $\{ x \in K : f(x) > a\}$ is convex for every $a \in L$. Notice that there is no need to introduce the dual notion of \textit{quasiconvex} map, for it simply coincides with that of convex map. However,  
a quasiconcave map may be non-concave (consider for instance $f$ defined on $\mathbb{B}$ by $f(0) = 1$, $f(1)=0$).

\begin{proposition}\label{thm:maxbauer}
Let $S$ be a topological semilattice, $K$ be a nonempty closed convex subset of $S$, and $L$ be a chain with a greatest element $\top$. Let $f : K  \rightarrow L$ be a quasiconcave map, $f \not\equiv \top$. We also suppose that $f$ is lower-compact, in the sense that the subset   
\[
\{ x \in K : f(x) \leqslant a \}
\] 
is compact for all $a \in L \setminus \{ \top \}$.  
Then $\argmin f$ is a face of $K$, and $f$ attains its minimum on $\ex K$. 
\end{proposition}

\begin{proof}
Let $x_0 \in K$ such that $f(x_0) \neq \top$. The subset $F = \{ f \leqslant f(x_0) \}$ is nonempty compact, so that $f$ attains its minimum on $F$, hence on $K$. 
Let $a := \min_{x \in K} f(x) < \top$. Then $\argmin f = \{ x \in K : f(x) \leqslant a\}$ is nonempty compact. With the quasiconcavity of $f$, $\argmin f$ is  an extreme subset of $K$. Thus, every minimal element of $\argmin f$ (which exists by Wallace's lemma) belongs to $\ex K$. 
\end{proof}

\section{Extension of the Krein--Milman theorem in semilattices}\label{sec:convsemilat2}

\subsection{Introduction}\label{subsec:intro}

It is natural to ask whether the Krein--Milman theorem also holds in locally compact closed convex subsets of some locally convex topological semilattice. As such, the answer is negative. For instance, the set $S = K = (-\infty, 0] \times (-\infty, 0]$ equipped with its usual (componentwise) semilattice structure and its usual topology is a locally convex, locally compact topological semilattice, but it has no extreme point. 

An additional hypothesis is certainly needed, and classical convex analysis helps to intuit it. Recall that, in 1957, Klee \cite[Theorem~3.4]{Klee57} notably improved the classical Krein--Milman theorem, for he showed that, in a locally convex Hausdorff topological vector space, every locally compact closed convex subset \textit{containing no line} is the closed convex hull of its extreme points and rays. 
In semilattices, the concept of extreme ray reduces to that of extreme point, but how could we define a suitable notion of line? 
Before coming to our proposal, we introduce \textit{locally convact semilattices}, where a \textit{convact} subset is a compact convex subset.

\subsection{Separation in locally convact semilattices}

A topological semilattice in which every element has a basis of convact neighbourhoods is called a \textit{locally convact} topological semilattice. 
This is equivalent to requiring the semilattice to be both locally convex and locally compact, since the topological closure of a convex subset remains convex. 

\begin{example}
If $X$ is a locally compact Hausdorff topological space, the \textit{upper space} $(U[X], \subset)$ of $X$ is the semilattice of all nonempty compact subsets of $X$ topologized with the Lawson topology. The term \textit{upper space} was coined by Edalat \cite{Edalat95}. 
Recall that $U[X]$ is a continuous semilattice \cite[Proposition~3.3]{Edalat95}, hence a strictly locally convex topological semilattice \cite[Exercise~III-2.17]{Gierz03}. It is also known that $U[X]$ is locally compact (see Liukkonen and Mislove \cite[Paragraph~I]{Liukkonen83}). 
\end{example}

\begin{problem}
By \cite[Proposition~7.1]{Lawson74} a locally compact Hausdorff semi\-topological group is topological. Is a locally convact Hausdorff semi\-to\-po\-lo\-gical semilattice with closed order necessarily a topological semilattice?
\end{problem}

The next lemma is implicit in the paper by Liukkonen and Mislove \cite{Liukkonen83}, but it deserves a specific statement. 

\begin{lemma}\label{lem:compact}
In a locally convact topological semilattice, every non\-empty relatively compact subset has a supremum, and every nonempty compact convex subset has a greatest element. 
\end{lemma}

\begin{proof}
For the second assertion see e.g.\ \cite[Proposition~VI-1.13(v)]{Gierz03} (it suffices for the ambient semilattice to be Hausdorff semitopological). 
Let $K$ be a locally convact topological semilattice, and $A$ be a nonempty relatively compact subset.  
Then $\overline{A}$ is compact, so by \cite[Lemma~5.2]{Lawson69} we can find a compact convex subset $C$ of $K$ containing $\overline{A}$. Then we know by the fundamental theorem that $C$ is a complete semilattice, so $A$ has a supremum $a_0 = \bigoplus_C A$ in $C$. We show that $a_0$ is also the supremum of $A$ in $K$. So let $x \in K$ be an upper bound of $A$ in $K$. Since $x \in \uparrow\!\! C$, the set $C \cap \downarrow\!\! x$ is nonempty compact convex, so it has a greatest element $c$. Then $a_0 \leqslant c \leqslant x$. This proves that $A$ has a supremum in $K$. 
\end{proof}

\begin{remark}
The previous proof actually uses the concept of projection. To see this, let $S$ be a locally convex topological semilattice, and $K$ be a nonempty compact convex subset of $S$. Then, for every $x \in \uparrow\!\! K$, the set $K \cap \downarrow\!\! x = \{ k \in K : k \leqslant x \}$ is nonempty compact convex so has a greatest element, so we can define the \textit{projection of $x$ on $K$} by 
\[
p_K(x) := \bigoplus_K \{ k \in K : k \leqslant x \}. 
\]
The partial map $p_K$ deserves to be called a projection for it satisfies $p_K \circ p_K = p_K$ and $p_K(x) \leqslant x$ for all $x \in \uparrow\!\! K$. 
Moreover, if $x \not\in K$, the set 
\[
H = \{ y \in S : y \leqslant x \Rightarrow y \leqslant p_K(x) \}
\]
is a halfspace (i.e.\ a convex subset with a convex complement) separating $K$ and $x$. Compare with Cohen et al.\ \cite[Theorem~8]{Cohen04}, where a similar statement is given for complete idempotent modules.  
\end{remark}

Now we can legitimately recall the results of Liukkonen and Mislove \cite[Proposition~1.1]{Liukkonen83}. 

\begin{proposition}\cite[Proposition~1.1]{Liukkonen83}\label{prop:liu}
Let $K$ be a locally compact topological semilattice. Then $K$ is locally convex if and only if the map $U[K] \ni A \mapsto \bigoplus_K A \in K$ is a continuous morphism. In this case, if $A \subset K$ is compact, then $A$ has a compact convex neighbourhood in $K$, and there is a minimal subset $B \subset A$ such that $\bigoplus_K A = \bigoplus_K B$. 
\end{proposition}

An additional ingredient will be needed for our advanced Krein--Milman type theorem, namely a result for separating the points in locally convact semilattices. This role is played by the following result. 

\begin{proposition}[Compare with \protect{\cite[Theorem~4.1]{Lawson69}}]\label{prop:sep2}
In a locally con\-vact topological semilattice $K$, 
let $A$ be a nonempty closed upper subset and $x \notin A$. Then there exists a continuous semilattice-morphism $\varphi : K \rightarrow [0,1]$ that commutes with arbitrary existing suprema and such that $\varphi(A)=\{1\}$ and $\varphi(x) = 0$. 
\end{proposition}

\begin{proof}
With a proof similar to that of Urysohn's lemma and utilizing the axiom of choice, Lawson \cite[Theorem~4.1]{Lawson69} built a continuous semilattice-morphism $\varphi : K \rightarrow [0,1]$ such that $\varphi(x) = 0$,  $\varphi(A) = \{1\}$, and of the form $\varphi(z) = \bigoplus \{ t : z \in V_t \}$, where $t$ runs over the set of dyadic numbers in $[0, 1]$, and $V_t = K \setminus \downarrow\!\! z_t$ for some $z_t \in K$. We show that $\varphi$ preserves arbitrary existing suprema. If $F$ is a nonempty subset of $K$ with supremum, then $\bigoplus F \notin V_t \Leftrightarrow \bigoplus F \leqslant z_t \Leftrightarrow (\forall f \in F) (f \leqslant z_t) \Leftrightarrow (\forall f \in F) (f \notin V_t)$. We deduce that $\{ t : \bigoplus F \in V_t \} = \bigcup_{f \in F} \{ t : f \in V_t\}$, so that $\varphi(\bigoplus F) = \bigoplus_{f \in F} \bigoplus \{ t : f \in V_t \} = \bigoplus_{f \in F} \varphi(f)$, i.e.\ $\varphi$ preserves existing suprema. 
\end{proof}

Note that, under the same hypothesis, if $x, y \in K$ such that $x \not\leqslant y$, this proposition provides a continuous semilattice-morphism $\varphi : K \rightarrow [0,1]$ that commutes with arbitrary nonempty suprema and such that $\varphi(x) = 1$ and $\varphi(y) = 0$. In particular, the $\varphi$'s separate the points of $K$. 
 
We state two additional results on separation in semilattices. They will not be used later on, but we believe they are of independent interest.  
 
\begin{proposition}\label{prop:sep3}
In a locally convact topological semilattice $K$, 
let $A$ be a compact convex subset and $x \notin A$. Then there exists an open convex neighbourhood $V$ of $A$ such that $x \notin \overline{V}$.  
\end{proposition}

\begin{proof}
If $x \notin \uparrow\!\! A$, then, considering that $\uparrow\!\! A$ is a closed (use e.g.\ \cite[Proposition~VI-1.6(ii)]{Gierz03}) and upper subset of $K$, we can apply Proposition~\ref{prop:sep2} and take $V = \{ \varphi > 1/2 \}$. 
Otherwise, $B := A \cap \downarrow\!\! x$ is nonempty compact convex, so its has a greatest element $b = \bigoplus_K  B \in B$.  
Since $x \notin A$, $x \neq b$. Thus $x \not\leqslant b$, so there is some $\psi : K \rightarrow [0,1]$ such that $\psi(x) = 1$ and $\psi(b) = 0$. Hence, the set $U := \{ \psi < 1/2 \}$ is open in $K$ and contains $B$. Now consider $C = A \setminus U$. For every $c \in C$, $c \not\leqslant x$, i.e.\ $x \notin \uparrow\!\! C$. But $C$ is closed in $A$, hence compact, so $\uparrow\!\! C$ is closed by \cite[Proposition~VI-1.6(ii)]{Gierz03}, and Proposition~\ref{prop:sep2} applies again: there is some $\varphi : K \rightarrow [0, 1]$ such that $\varphi(C) = \{1\}$ and $\varphi(x) = 0$. To conclude the proof, choose $V = \{ \psi < 1/2 \} \cup \{ \varphi > 1/2 \}$. This is an open convex subset containing $A$, and $x \notin \overline{V} \subset \{ \psi \leqslant 1/2 \} \cup \{ \varphi \geqslant 1/2 \}$.  
\end{proof}

\begin{remark}
In the proof the convex set $\{ \psi \leqslant 1/2 \} \cup \{ \varphi \geqslant 1/2 \}$ has a convex complement, i.e.\ is a (closed) halfspace. 
\end{remark}

\begin{corollary}[Axiom $\nsq$]
In a compact locally convex topological semilattice $K$, 
let $A, B$ be disjoint closed convex subsets. Then there exists a closed convex neighbourhood of $A$ that is disjoint from $B$. 
\end{corollary}

\begin{proof}
Let $x \in B$. Since $A$ is disjoint from $B$, $x \notin A$, hence there exists some open convex neighbourhood $V_x$ of $A$ such that $x \notin \overline{V}_x$ by Proposition~\ref{prop:sep3}. The family of open subsets $(K \setminus \overline{V}_x)_{x \in B}$ covers the compact subset $B$, hence admits a finite subcover $(K \setminus	 \overline{V}_x)_{x \in F}$, with $F \subset B$ finite. Therefore, $K \setminus B \supset \bigcap_{x \in F} \overline{V}_x \supset \bigcap_{x \in F} V_x \supset A$. Thus, $\bigcap_{x \in F} \overline{V}_x$ is a closed convex neighbourhood of $A$ that is disjoint from $B$. 
\end{proof}

\subsection{Extension to the locally compact case}

To resolve the problem raised in the introduction (Paragraph~\ref{subsec:intro}),  
we define a \textit{line} of a topological semilattice $S$ as an upper-bounded chain in $S$ that is not relatively compact. Hence a closed subset $K$ of $S$  \textit{contains no line} if every upper-bounded chain in $K$ is contained in a compact subset of $K$. 

\begin{lemma}\label{lem:klee}
Let $S$ be a locally convex topological semilattice, and let $K$ be a locally compact closed convex subset of $S$ containing no line. Then every element $x$ of $K$ is the supremum in $K$ of the extreme points of $K$ below $x$. 
\end{lemma}

One may find some similarities between the following proof and that of 
the tropical analogue of Minkowski's theorem in $\mathbb{R}_+^n$ (see Gaubert and Katz \cite[Theorem~3.2]{Gaubert07} and Butkovic, Schneider, and Sergeev \cite[Proposition~24]{Butkovic07}, see also Helbig \cite[Theorem~IV.5]{Helbig88} for a first but less precise statement, and Develin and Sturmfels \cite[Proposition~5]{Develin04} for an analogue of Carath\'eorory's theorem).  

\begin{proof}
If $\varphi : K \rightarrow [0, 1]$ is a semilattice-morphism, we let $K_{\varphi} = \{ u \in K : u \leqslant x, \varphi(u) = \varphi(x) \}$. Let $C$ be a maximal chain in $K_{\varphi}$ (containing $x$). Then $C$ is upper-bounded, contained in $K$, but must not be a line, hence is relatively compact. Since a maximal chain in a poset with a semiclosed topology is always closed \cite[Proposition~VI-5.1]{Gierz03}, and since $K_{\varphi}$ is closed, we deduce that $C$ is compact. In particular, $C$ has a least element $u_{\varphi} \in C$ by Wallace's lemma.   
We show that $u_{\varphi}$ is an extreme point of $K$. If there are $v, w \in K$ such that $u_{\varphi} = v \oplus w$, then $\varphi(x) = \varphi(u_{\varphi}) = \max(\varphi(v), \varphi(w))$. Let us assume, without loss of generality, that $\varphi(x) = \varphi(v)$. It follows that $v \in K_{\varphi}$. Also, $u_{\varphi} \geqslant v$, so that $u_{\varphi} = v$ by definition of $u_{\varphi}$. This proves that $u_{\varphi} \in \ex K$. 

Now let $y \in K$ be some upper bound of the set $\ex K \cap \downarrow\!\! x$ in $K$. Then $y \geqslant u_{\varphi}$ for all $\varphi$, so that $\varphi(y) \geqslant \varphi(u_{\varphi}) = \varphi(x)$ for all $\varphi$. By Proposition~\ref{prop:sep2}, this implies that $y \geqslant x$. This proves that $x = \bigoplus_K \ex K \cap \downarrow\!\! x$. 
\end{proof}

\begin{remark}
We can be more restrictive in the definition of a line. 
Redefine a line in $S$ as an upper-bounded chain $C$ that is not relatively compact and that satisfies $\downarrow\!\! c \not\subset C$, for all $c \in C$. One can check that the previous proof still works. Consequently, the lemma now encompasses the case where the set $S$ is itself a chain (considered as a locally convex topological semilattice when equipped with its interval topology). 
\end{remark}

Let $S$ be a locally convex topological semilattice. 
If $\mathit{\Psi}$ is the set of continuous semilattice-morphisms $\psi : S \rightarrow [0, 1]$, there is a natural mapping $S \rightarrow [0, 1]^{\mathit{\Psi}}$. This is not an injective map in general, for the $\psi$'s do not necessarily separate the points of $S$. 
If we equip the set $[0, 1]^{\mathit{\Psi}}$ with the (compact Hausdorff) product topology, which amounts to the topology of pointwise convergence, we define the \textit{weak topology} $\sigma(S, \mathit{\Psi})$ as the topology on $S$ generated by the family 
\[
\{ \psi^{-1}(V) : \psi \in \mathit{\Psi}, V \mbox{ open in } [0, 1] \}. 
\]
It is coarser than the original topology. 
Moreover, if $K$ is a subset of $S$, then the topology induced on $K$ by $\sigma(S, \mathit{\Psi})$ coincides with $\sigma(K, \mathit{\Psi}|_K)$, where $\mathit{\Psi}|_K$ denotes the family of restrictions of the functions in $\mathit{\Psi}$ to $K$ (see \cite[Lemma~2.53]{Aliprantis06}). 
Note that, if $K$ is a locally compact closed convex subset of $S$, 
we cannot conclude that $\sigma(K, \mathit{\Psi}|_K)$  coincides with the original topology (one would like to use e.g.\ \cite[Theorem~2.55]{Aliprantis06}) because the $\psi$'s restricted to $K$ do not separate points and closed subsets in general. 
We now restate the Klee--Krein--Milman type theorem given in the Introduction (Theorem~\ref{thm:klee0}). 

\begin{theorem}[Klee--Krein--Milman for semilattices]\label{thm:klee}
In a locally con\-vex topological semilattice, every locally compact weakly-closed convex subset containing no line is the weakly-closed convex hull of its extreme points. 
\end{theorem}

\begin{proof}
Let $S$ be a locally con\-vex topological semilattice, let $K$ be a locally compact weakly-closed convex subset of $S$ containing no line, and let $x \in K$. Since $K$ is closed in the weak topology, it is closed in the original topology, so the previous lemma applies: we have $x = \bigoplus_K D$, where $D$ is the directed subset $\co(\ex K \cap \downarrow\!\! x)$. We have to show that $x$ is in the weak closure $D^*$ of $D$. 
So assume that $x \notin D^*$. By definition of the weak topology, there are open subsets $V_1, \ldots, V_k$ of $[0,1]$ and continuous semilattice-morphisms $\psi_1, \ldots, \psi_k : S \rightarrow [0, 1]$ such that $x \in \bigcap_{j = 1}^k \psi_j^{-1}(V_j) \subset S \setminus D^*$. Let us denote by $\varphi_1, \ldots, \varphi_k$ the respective restrictions of $\psi_1, \ldots, \psi_k$ to $K$. Using the notations of the proof of Lemma~\ref{lem:klee}, we let $u = u_{\varphi_1} \oplus \ldots \oplus u_{\varphi_k}$. Then $u$ is in $D$ as a finite join of extreme points of $K$ below $x$. Remembering that $\varphi_j(u_{\varphi_j}) = \varphi_j(x)$ for all $j$, one can see that $\varphi_j(u) = \varphi_j(x)$ for all $j$. This implies that $\varphi_j(u) \in V_j$ for all $j$, thus $u \in S \setminus D^*$, a contradiction. 
\end{proof}

\subsection{Milman's converse}

In Section~\ref{sec:or} we have proved Milman's theorem in po\-spaces with the lower, upper, or order convexity. For topological semilattices, this result is less evident, since the convex hull of a compact subset does not need to be closed in general. 
Fortunately, it does work, even for locally compact convex subsets.  
The next lemma is interesting in its own right.  

\begin{lemma}\label{lem:ch}
Let $S$ be a locally convex topological semilattice, and $K$ be a locally compact convex subset of $S$. Then, for every compact subset $A$ of $K$ and every $x \in \ex K$, $x = \bigoplus_K A$ implies $x \in A$.  
\end{lemma}

\begin{proof}[First proof]
By Proposition~\ref{prop:liu}, there is a minimal subset $B$ of $A$ such that $\bigoplus_K B = \bigoplus_K A$. 
If $B$ is empty or a singleton, then $x \in A$ is clear. 
Otherwise, let $b\in B$. Then $B \setminus\{b\}$, as a nonempty relatively compact subset of $A$, has a supremum $b_0$ in $K$ by Lemma~\ref{lem:compact}. 
Moreover, $x = \bigoplus_K B = b_0 \oplus b$. Since $x \in \ex K$, we get $x \in \{ b_0, b\}$. By minimality of $B$, $x \neq b_0$, so $x = b \in A$. 
\end{proof}

\begin{proof}[Second proof]
Assume that $x \not\in A$. One may wish to apply Proposition~\ref{prop:sep3}, but here we do not assume $A$ to be convex. 
For every $a \in A$, $a \not\geqslant x$, and by Proposition~\ref{prop:sep2} there exists a continuous semilattice-morphism $\varphi_a : K \rightarrow [0,1]$ such that $\varphi_a(a) = 0$ and $\varphi_a(x) = 1$. 
Let $V_a$ be the open subset $\{ \varphi_a < 1/2\}$ of $K$. The compact set $A$ is covered by the open family $\{ V_a \}_{a \in A}$, so we can extract a finite subfamily $\{ V_a \}_{a \in F}$ still covering $A$. If $H_a := \{ \varphi_a \leqslant 1/2\}$, we deduce that $A = \bigcup_{a \in F} (A \cap H_a)$. Every $A \cap H_a$ is compact and can be supposed nonempty, hence has a supremum in $K$ by Lemma~\ref{lem:compact}, thus $x = \bigoplus_K A = \bigoplus_{a \in F} (\bigoplus_K A \cap H_a)$. But $x$ is an extreme point of $K$, so that $x = \bigoplus_K (A \cap H_{a_0})$ for some $a_0 \in F$. Proposition~\ref{prop:sep2} also says that $\varphi_{a_0}$ can be choosen so as to preserve arbitrary nonempty suprema in $K$, so $1 = \varphi_{a_0}(x) = \varphi_{a_0}(\bigoplus_K A \cap H_{a_0}) = \bigoplus \varphi_{a_0}(A \cap H_{a_0}) \leqslant 1/2$, a contradiction. 
\end{proof}

\begin{remark}
Compare Lemma~\ref{lem:ch} with \cite[Corollary~V-1.4]{Gierz03}, which is a similar result that holds in continuous lattices. See also \cite[p.\ 403]{Gierz03}. 
\end{remark}

\begin{theorem}[Milman for semilattices]
Let $S$ be a locally convex topological semilattice, and $K$ be a locally compact closed convex subset of $S$. Then, for each compact subset $A$ of $K$ such that $K = \overline{\co}(A)$, we have $A \supset \ex K$. 
\end{theorem}

\begin{proof}
Let $x \in \ex K$, and assume that $x \not\in A$. Let $B := A \cap \downarrow\!\! x$, and suppose at first that $B$ is nonempty. Then $B$ is nonempty compact, so admits a supremum $b = \bigoplus_K B$ in $K$ by Lemma~\ref{lem:compact}. Moreover, $x \neq b$ by the preceding lemma. 
Now the same method used in the proof of Proposition~\ref{prop:sep3} provides a closed convex neighbourhood $\overline{V}$ of $A$ such that $x \notin \overline{V}$. But $\overline{V} \supset \overline{\co}(A) = K$, a contradiction. 
If $B$ is empty, then $x \notin \uparrow\!\! A$, so with Proposition~\ref{prop:sep2} we can separate $x$ and the closed upper subset $\uparrow\!\! A$ by a continuous semilattice-morphism, and the same contradiction appears. 
\end{proof}

\section{Semilattices with finite breadth}\label{sec:fb}

\subsection{Breadth and Minkowski's theorem}

In locally convex topological semilattices, an important subclass is that of topological semilattices with finite breadth. The \textit{breadth} is defined as the least integer $b$ such that, for all nonempty finite subsets $F$, there exists some $G \subset F$ with at most $b$ elements such that $\bigoplus F = \bigoplus G$. It turns out that the breadth has a direct geometric interpretation, for as noticed by Jamison \cite[Paragraph~4.D]{Jamison82} it coincides with the Carath\'eodory number of the semilattice equipped with its algebraic convexity. 
The next lemma prepares a series of results on topological semilattices with finite breadth. 

\begin{lemma}\label{lem:b}
Let $S$ be a topological semilattice with finite breadth $b$.  
If $A$ is a compact subset of $S$, so is $\co(A)$. 
\end{lemma}

\begin{proof}
First remark that $A_b := \{ x_1 \oplus \ldots \oplus x_b : x_1, \ldots, x_b \in A \}$ is a set between $A$ and $\co(A)$. Moreover, this is a semilattice by definition of breadth, hence $\co(A) = A_b$. 
This also means that $\co(A)$ is the image of $A \times \ldots \times A$ by the continuous map $\phi : S \times \ldots \times S \rightarrow S$, $(x_1, \ldots, x_b) \mapsto x_1 \oplus \ldots \oplus x_b$. So if $A$ is compact, $\co(A) = \phi(A \times \ldots \times A)$ is compact. 
\end{proof}

The following result is due to Lawson \cite[Theorem~1.1]{Lawson71}; it is a consequence of Lemma~\ref{lem:b}. 

\begin{proposition}[Lawson]
Every topological semilattice with finite breadth $b$ is locally convex. 
\end{proposition}

\begin{proof}
Let $G$ be an open subset containing some point $x$. The continuity of $\phi$ defined above and the fact that $\phi(x, \ldots, x) \in G$ imply that $x \in V \subset V_b \subset G$ for some open subset $V$, where $V_b := \{x_1 \oplus \ldots \oplus x_b : x_1, \ldots, x_b \in V \}$. Thus $V_b = \co(V)$ is a convex neighbourhood of $x$ contained in $G$. 
\end{proof}

A topological semilattice has \textit{compactly finite breadth} if every nonempty compact subset $A$ contains a finite subset $F$ with $\bigoplus A = \bigoplus F$. See Liukkonen and Mislove \cite[Theorem~1.5]{Liukkonen83} for equivalent conditions in locally convact topological semilattices, and Lawson et al.\ \cite[Theorem~1.11]{Lawson85} for additional	conditions. 
Another consequence of Lemma~\ref{lem:b} is that ``finite breadth'' is stronger than ``compactly finite breadth'' in a locally compact semilattice. 

\begin{corollary}
Every locally compact topological semilattice with finite breadth has compactly finite breadth. 
\end{corollary}

\begin{proof}
If $A$ is a nonempty compact subset, then $A$ has a supremum $a$ by Lemma~\ref{lem:compact}, and $a \in \overline{\co}(A) = \co(A)$ by Lemma~\ref{lem:b}, so that $a = \bigoplus F$ for some finite $F \subset A$. 
\end{proof}

A semilattice is \textit{distributive} if, for all $x, y, z \in S$ with $x \leqslant y \oplus z$, there exists some $y' \leqslant y$ and $z' \leqslant z$ such that $x = y' \oplus z'$. Also recall that a (distributive) \textit{lattice} is a (distributive) semilattice in which every nonempty finite subset has an infimum. 

\begin{theorem}
In a topological distributive lattice $S$ with finite breadth $b$ (still equip\-ped with the algebraic semilattice convexity), let $K$ be a compact convex subset of $S$. Then every $x \in K$ can be written as the convex combination of at most $b$ extreme points. 
\end{theorem}

\begin{proof}
Let $L$ be the lattice generated by $K$ in $S$. By Lemma~\ref{lem:b} (applied to $S$ and $L$ with the opposite order), $L$ is compact, so this is a compact locally convex topological semilattice. 
Using either \cite[Proposition~1.13.3]{vanDeVel93} or a combination of \cite[Theorem~III-2.15]{Gierz03} and  the proof of \cite[Proposition~III-2.13]{Gierz03}, one can assert that $L$ is also locally convex with respect to the order convexity. Thus, \cite[Theorem~3.1]{Stralka70}, due to Stralka, can be applied: $L$ as a topological lattice can be embedded (algebraically and topologically) in a product of $b$ compact (connected) chains $C = \prod_{j=1}^b C_j$. As a consequence, $K$ as a topological semilattice also embeds in $C$. For each $j = 1, \ldots, b$, we denote by $\varphi_j : K \rightarrow C_j$ the $j$th projection, which is a continuous semilattice-morphism. 

The remaining part of the proof can now mimic that of Lemma~\ref{lem:klee}, using the finite collection of maps $\{ \varphi_j : j = 1, \ldots, b \}$, which separates the points of $K$, instead of the whole collection of continuous semilattice-morphisms $\varphi : K \rightarrow [0, 1]$. This leads to the fact that, for all $x \in K$, one can write $x = u_1 \oplus \ldots \oplus u_b$ for some extreme points $u_1, \ldots, u_b$ of $K$. 
\end{proof}

\begin{problem}
Does the conclusion of this theorem still hold for $S$ a topological distributive \textit{semilattice} with finite breadth? 
\end{problem}

As a final remark, it should be emphasized that, in a locally convex topological semilattice $S$, the set $\ex K$ of extreme points of some compact convex subset $K$ is not necessarily closed. Actually, if $S$ is distributive, it is known that $\ex K$ is closed if and only if the way-above relation on $K$ is additive \cite[Proposition~V-3.7]{Gierz03}.

\subsection{Depth of a semilattice}

The \textit{depth} of a semilattice, defined as the supreme cardinality of a chain, is another important convex invariant, as highlighted by the following result\footnote{This result is left as an exercise in \cite[Exercise~II-1.23]{vanDeVel93}. As far as we know, no proof of it exists in the literature. }. Recall that the Helly number is the least integer $h$ such that each finite family of convex subsets meeting $h$ by $h$ has a nonempty intersection. 

\begin{proposition}\label{prop:depth1}
The Helly number of a semilattice equals its depth. 
\end{proposition}

To prove this assertion, we shall need a result due to Jamison \cite[Theorem~7]{Jamison81}, which says that in a finite \textit{convex geometry} (see the definition in Paragraph~\ref{ssec:icg}), the Helly number equals the clique number, so first we give some definitions. Let $X$ be a convexity space. A subset $K$ of $X$ is \textit{free} if it is both convex and \textit{independent}, i.e.\ such that $K = \ex K$. A \textit{clique} is a maximal free subset, and the \textit{clique number} of $X$ is the supremum of the cardinalities of all cliques. 	

\begin{lemma}
The free subsets (resp.\ the cliques) of a semilattice coincide with its chains (resp. its maximal chains), and the clique number of a semilattice equals its depth. 
\end{lemma}

\begin{proof}
Let $C$ be a free subset of a semilattice, let $x, y \in C$, and let us prove that $x$ and $y$ are comparable. Since $C$ is convex, $z := x \oplus y \in C$. But $C = \ex C$, so $z$ is a extreme point of $C$, hence $z \in \{ x, y\}$, i.e.\ $x \leqslant y$ or $y \leqslant x$. This proves that $C$ is a chain. The converse statement is straightforward, and the rest of the proof follows. 
\end{proof}

\begin{proof}[Proof of Proposition~\ref{prop:depth1}]
Write $d$ for the depth of $S$. Let $n$ be an integer $\leqslant d$, and let $C$ be a chain with cardinality $n$. Then the finite family $(K_c)_{c \in C}$ of convex subsets $K_c = C \setminus \{c\}$ meets $n-1$ by $n-1$ but is of empty intersection, so that $h > n-1$. This implies that $h \geqslant d$ (even if $d = \infty$). If $d = \infty$, we get $h = d$. 

Now assume that $d$ is finite.  
Let $(K_j)_{j\in J}$ be a finite family of convex subsets meeting $d$ by $d$. 
For every $I \subset J$ with cardinality $d$, let $x_I \in \bigcap_{j \in I} K_j$. Denote by $X$ the subsemilattice of $S$ generated by $\{x_I \}_{I \subset J, |I|=d}$. Note that the depth $d_X$ of $X$ is less than $d$. Moreover, $X$ is a finite set, and the algebraic convexity on $X$ is a \textit{convex geometry} (see the definition in Paragraph~\ref{ssec:icg}). Thus, \cite[Theorem~7]{Jamison81} applies, i.e.\ the clique number of $X$ equals its Helly number $h_X$. By the previous lemma, this rewrites to $h_X = d_X$, hence $h_X \leqslant d$. 
Now, let $X_j$ be the subsemilattice of $X$ generated by $\{ x_I \}_{I \subset J, |I|=d, j \in I}$. 
Then $x_I \in \bigcap_{j \in I} X_j$ for all $I$, so that $(X_j)_{j\in J}$ is a finite family of convex subsets of $X$ meeting $d$ by $d$. Since $h_X \leqslant d$, we have $\bigcap_{j \in J} X_j \neq \emptyset$, by definition of the Helly number. Morever, $X_j \subset K_j$, so we get $\bigcap_{j \in J} K_j \neq \emptyset$. This shows that $h \leqslant d$. 
\end{proof}

The next result connects the depth with the extreme points of convex subsets and can be seen as a corollary of Lemma~\ref{lem:klee}. 

\begin{proposition}\label{prop:depth2}
Let $S$ be a locally convex topological distributive semilattice, and $K$ be a locally compact closed convex subset of $S$. Assume that $K$ has finite depth $d$. Then $K$ is finite and has exactly $d$ extreme points. 
\end{proposition}

\begin{proof}
We follow the proof given by Blyth \cite[Theorem~5.3]{Blyth05} for finite distributive lattices. Let $C$ be a chain of maximal length $d$ in $K$. For convenience, we write $c_1 < \ldots < c_d$ for elements of $C$. Let $\theta : \ex K \rightarrow C$ such that $\theta(p) = \min\{ c\in C : c \geqslant p \}$. 
Note that $c_1$ is necessarily the least element of $K$, hence is in $\ex K$, and $\theta(c_1) = c_1$. 
If $c_k \in C \setminus \{ c_1 \}$, there exists some $p \in \ex K$ such that $p \leqslant c_k$ and $p \not\leqslant c_{k-1}$, since $\ex K$ order-generates $K$ by Lemma~\ref{lem:klee}. This implies $\theta(p) = c_k$. We have shown that $\theta$ is surjective. 

Let us prove that $\theta$ is injective. Assume that $\theta(p) = \theta(q) = c_k \in C$ for some $p, q \in \ex K$. If $c_k = c_1$, then $p = q = c_k$, so suppose that $c_k \neq c_1$.  
Then $c_{k-1} \oplus p \leqslant c_k$ is clear, and one also has $c_{k-1} \oplus p \geqslant c_k$, otherwise $c_{k-1} < c_{k-1} \oplus p < c_k$ which is impossible because of the maximality of $C$. We get $c_{k-1} \oplus p = c_k$, and symmetrically $c_k = c_{k-1} \oplus q$. 
Thus, $p \leqslant c_{k-1} \oplus p = c_{k-1} \oplus q$. The distributivity of $S$ and the fact that $p$ is an extreme point of $K$ imply $p \leqslant c_{k-1}$ (which would contradict $\theta(p) = c_k$) or $p \leqslant q$. Similarly, $p \geqslant q$, so $p=q$, and $\theta$ is injective, hence bijective. This proves that the cardinality of $\ex K$ equals $d$. 

Since $K$ has finite depth, every (upper-bounded) chain in $K$ is finite hence compact, so $K$ contains no line. 
By Lemma~\ref{lem:klee}, the finite subset $\ex K$ order-generates $K$, so that $K$ is finite. 
\end{proof}

\section{Convex geometries on semilattices and lattices}\label{sec:cg}

\subsection{Introduction}\label{ssec:icg}

Some convexities may not satisfy a Krein--Milman type theorem and, for some of them, even polytopes may not coincide with the convex hull of their extreme points. This last property actually characterizes convexities that are convex geometries, whose usual definition follows. 
A convexity space $X$ is a \textit{convex geometry} (or an \textit{antimatroid}) if, given a convex subset $K$, and two unequal points $x$ and $y$, neither in $K$, then $y \in \co(K \cup \{x\})$ implies $x \notin \co(K \cup \{y\})$. This amounts to say that the relation $\leqslant_K$ defined on $X \setminus K$ by $x \leqslant_K y \Leftrightarrow y \in \co(K \cup \{x\})$ is a partial order.  
The convexities previously introduced, namely the order (resp.\ lower, upper) convexity for posets, and the algebraic convexity for semilattices, are indeed convex geometries (see \cite[Exercise~I-2.24]{vanDeVel93}). 
In this section, we investigate some other convexities on semilattices and lattices that are not convex geometries in general. 

Let $X$ be a convexity space and $x \in X$. A \textit{copoint} at $x$ is a convex set $C \subset X$ maximal with the property $x \notin C$, in which case $x$ is an \textit{attaching point} of $C$. 

\begin{lemma}\label{lem:cop}
Let $X$ be a convexity space. If $C$ is a convex subset and $x \notin C$, there is some copoint at $x$ containing $C$. 
\end{lemma}

\begin{proof}
This is an easy consequence of Zorn's lemma. 
\end{proof}

The next important theorem, due to Jamison \cite{Jamison80}, and to Edelman and Jamison \cite{Edelman85} for the case where the set $X$ is finite, lists several equivalent conditions for a convexity to be a convex geometry. For the sake of completeness, we shall give a proof of this result. 

\begin{theorem}[Jamison--Edelman]\label{thm:cg}
Let $X$ be a convexity space. Then the following are equivalent: 
\begin{enumerate}
	\item\label{cg1} $X$ is a convex geometry, 
	\item\label{cg2} each polytope is the convex hull of its extreme points, 
	\item\label{cg3} for each copoint $C$ at $x$, the set $C \cup \{x\}$ is convex,  
	\item\label{cg4} each copoint $C$ has a unique attaching point.  
\end{enumerate}
\end{theorem}

\begin{proof}
(\ref{cg1}) $\Rightarrow$ (\ref{cg3}). Assume that $X$ is a convex geometry, and let $C$ be a copoint at $x$. Assume that $C \cup \{x\}$ is not convex, i.e.\ there is some $y \in \co(C \cup \{x\})$, $y \notin C \cup \{x \}$. Then $\co(C \cup \{y\})$ is a convex set avoiding $x$ and strictly greater than $C$, a contradiction. 

(\ref{cg3}) $\Rightarrow$ (\ref{cg4}). Let $C$ be a copoint at $x$, and assume that it has another attaching point $y \neq x$. Then, by $(\ref{cg3})$, $C \cup \{y\}$ is a convex set avoiding $x$ and strictly greater than $C$, a contradiction. 

(\ref{cg4}) $\Rightarrow$ (\ref{cg2}). Let $K$ be a polytope, and let $F$ be a minimal finite subset such that $K = \co(F)$. Consider some $x \in F$ that is not an extreme point of $K$. By minimality of $F$, $x \notin \co(F \setminus \{x\})$, so there is some copoint $C$ at $x$ containing $\co(F \setminus \{x\})$. Since $x$ is not an extreme point, $C$ is strictly contained in $K \setminus \{x\}$, so there is some $y \neq x$, $y \notin C$. Let $D$ be a copoint at $y$ containing $C$. If $x \notin D$, then $C = D$ by maximality of $C$, but then, by (\ref{cg4}), $x = y$, a contradiction. Hence, $x \in D$, so that $D = K$, which contradicts $y \notin D$. So we have shown that $F \subset \ex K$, i.e.\ $K = \co(\ex K)$. 

(\ref{cg2}) $\Rightarrow$ (\ref{cg1}). Assume that, for some $x \neq y$ and some convex subset $K$, $x \in \co(K \cup \{y \}) \setminus K$ and $y \in \co(K \cup \{x\}) \setminus K$. It is easy to see that there exists some finite subset $F \subset K$ such that $x \in \co(F \cup \{y\})$ and $y \in \co(F \cup \{x\})$. Then the polytope $L = \co(F \cup \{x\}) = \co(F \cup \{y\})$ is the convex hull of its extreme points $\ex L$, and we deduce $\ex L \subset F \cup \{x\}$ and $\ex L \subset F \cup \{y\}$, hence $\ex L \subset F$, so that $L = \co(\ex L) \subset \co(F) \subset K$. This contradicts $x \notin K$. 
\end{proof}

For one more equivalent condition using the concept of \textit{meet-distributive lattice}, see Edelman \cite[Theorem~3.3]{Edelman80}, Birkhoff and Bennett \cite{Birkhoff85}, and Monjardet \cite{Monjardet85}. 

In the following paragraphs, we say that a topological convexity space \textit{satisfies the Krein--Milman property} if every compact convex subset is the closed convex hull of its extreme points. 

\subsection{The ideal convexity of a semilattice}

Recall from Section~\ref{sec:convsemilat} that the \textit{ideal convexity} of a semilattice consists of its lower subsemilattices. An element of a convex subset $K$ is then an extreme point of $K$ if and only if it is at the same time maximal and coprime in $K$ ($x$ is \textit{coprime} if, for every nonempty finite subset $F$ with $x \leqslant \bigoplus F$, $x \leqslant f$ for some $f \in F$). We call \textit{max-coprime} an element that is both maximal and coprime. 

\begin{proposition}\label{prop:fc}
A semilattice with the ideal convexity is a convex geometry if and only if it is a chain. 
In this case, when endowed with a compatible topology, it satisfies the Krein--Milman property. 
\end{proposition}

\begin{proof}
For a chain, the ideal convexity coincides with the lower convexity, thus is a convex geometry. The Krein--Milman property is then the terms of Theorem~\ref{thm:kmp}. 

Now assume that the ideal convexity of some semilattice $S$ is a convex geometry, and let us show that $S$ is a chain. So let $x, y \in S$ with $x \not\leqslant y$. Then $x \notin \downarrow\!\! y$, which is a convex subset. Thus, by Lemma~\ref{lem:cop}, there is some copoint $C$ at $x$ containing $\downarrow\!\! y$. By Theorem~\ref{thm:cg}, $C \cup \{ x \}$ is convex, and $y \in C$, so we have $y \oplus x \in C \cup \{ x \}$, i.e.\ $y \oplus x  \in C$ or $y < x$. The former case has to be rejected, otherwise $x \in \downarrow\!\!(y\oplus x) \subset C$. Hence, $y < x$, which concludes the proof. 
\end{proof}

We seize the opportunity to mention here that the ideal convexity was considered by Martinez \cite{Martinez72}, whose main result \cite[Theorem~1.2]{Martinez72} can be rephrased in the langage of abstract convexity as follows: 

\begin{theorem}[Martinez]
Consider a semilattice with the ideal convexity. Then the following are equivalent: 
\begin{itemize}
	\item the ideal convexity is completely distributive, 
	\item each copoint admits an attaching point with a unique copoint,  
	\item each element can be uniquely decomposed as the join of a finite number of pairwise incomparable coprime elements. 
\end{itemize}
\end{theorem}

Decomposing elements as joins of coirreducible or coprime elements has been the subject of a great amount of research in order theory (see e.g.\ Ern\'e \cite{Erne91, Erne06} and references therein, see also Bi\'nczak et al.\ \cite[Theorem~5.4]{Binczak07} on presentable semilattices), and this theorem invites us to look at these past results from an abstract convexity point of view. 

\begin{remark}
Martinez' theorem actually characterizes semilattices that are free $\mathbb{B}$-modules. Indeed, consider in the following lines a semilattice $S$ with a least element $0$, and assume for convenience that $S \neq \{ 0 \}$. 
The last condition in Martinez' theorem says that a subset of the family of coprime elements is a basis (i.e.\ a subset $B$ such that, for every $x$ there is a unique finite -possibly empty- subset of $B$ whose join is $x$). Conversely, assume that the semilattice admits a basis $B$, and let us show that every $b \in B$ is a non-zero coprime element. So let $F$ be a finite subset such that $b \leqslant \bigoplus F$. For all $x \in F$, there is a finite subset $F_x$ of $B$ such that $x = \bigoplus F_x$. Hence, $F' := \bigcup_{x \in F} F_x$ is a finite subset of $B$ whose join is $\bigoplus F$. Since $b \leqslant \bigoplus F$, $F' \cup \{b\}$ is another such subset, so $F' = F' \cup \{b\}$ by definition of $B$. This gives $b \in F'$, i.e.\ $b \in F_x$ for some $x \in F$. This shows that $b \leqslant x$ for some $x \in F$, i.e.\ that $b$ is a coprime element. Also, $b$ is non-zero, otherwise $0 \in B$ would be the join of both the empty set and $\{0\}$. 

Another consequence is that every semilattice that is a free $\mathbb{B}$-module is distributive. For suppose that $x \leqslant y \oplus z$, and let $F$ be a finite subset of a basis $B$ such that $x = \bigoplus F$. Since every element of $F$ is coprime, we have $f \leqslant y$ or $f \leqslant z$ for all $f \in F$. Then, if $y' = \bigoplus \{ f\in F : f \leqslant y \}$ and $z' = \bigoplus \{ f \in F : f \leqslant z \}$, we get $y' \leqslant y$, $z' \leqslant z$, and $x = y' \oplus z'$, which shows distributivity.  

Therefore, if a semilattice $S$ is a free $\mathbb{B}$-module, then it has a unique basis, equal to the subset of its non-zero coprime elements. To see this, let $B$ be a basis of $S$. Since $S$ is distributive, the subset of its coprime elements is $\ex S$, and we have seen that $B \subset (\ex S) \setminus \{ 0\}$. If $x \in (\ex S) \setminus \{0\}$, there exists a nonempty finite subset $F$ of $B$ such that $x = \bigoplus F$. Since $x$ is an extreme point of $S$, we deduce $x \in F$, so that $x \in B$. 

If now we define the \textit{rank} $r$ of a distributive semilattice $S$ as the cardinality of $(\ex S) \setminus \{ 0\}$, then, applying Proposition~\ref{prop:depth2} to $S$ equipped with the discrete topology, one can see that the following conditions are equivalent: 
\begin{itemize}
	\item $S$ is a $\mathbb{B}$-module of finite type, 
	\item $S$ has finite depth, 
	\item $S$ has finite rank, 
	\item $S$ is finite. 
\end{itemize}
In this case,  
the depth $d$ of $S$ equals $r+1$. Moreover, if $S$ is free, then $S$ is in bijection with the collection of subsets of $(\ex S) \setminus \{0\}$, hence has exactly $2^{r}$ elements. 
\end{remark}

\subsection{The order-algebraic convexity of a semilattice}

Quite different from the previous case is the one of the \textit{order-algebraic convexity} of a semilattice, made up of its order-convex subsemilattices, for it involves trees instead of chains. A \textit{tree} is a semilattice in which every principal filter $\uparrow\!\! x$ is a chain. It is an easy task to see that the set of extreme points of a convex subset is the union of its minimal elements and max-coprime elements. 
 
\begin{proposition}
A semilattice with the order-algebraic convexity is a convex geometry if and only if it is a tree.  
In this case, when endowed with a Hausdorff semitopological topology, it satisfies the Krein--Milman property. 
\end{proposition}

\begin{proof}
Assume that the order-algebraic convexity of some semilattice $S$ is a convex geometry, and let us show that $S$ is a tree. 
So let $a, b, x \in S$ such that $a \geqslant x$ and $b \geqslant x$. We want to prove that $a$ and $b$ are comparable, so suppose that $b \not\leqslant a$, i.e.\ $b \notin \downarrow\!\! a$. The subset $\downarrow\!\! a$ is convex, so by Lemma~\ref{lem:cop} there exists some copoint $C$ at $b$ containing $\downarrow\!\! a$. In particular, $a, x \in C$. 
Now use the fact that the convexity is a convex geometry: this implies that $C \cup \{ b \}$ is convex (Theorem~\ref{thm:cg}), hence $a \oplus b \in C \cup \{ b \}$. 
If $a \oplus b \in C$, then $b \in [x, a \oplus b] \subset C$, whereas $b \notin C$. Thus, $a \oplus b \in \{ b \}$, i.e.\ $b \geqslant a$. 

Conversely, consider a Hausdorff semitopological tree, and let $K$ be a compact convex subset. We (implicitly) follow the suggestion of proof from \cite[Exercise~I-5.26]{vanDeVel93}. Denote by $\leqslant_b$ the relation $\leqslant_{\{ b \}}$ defined on $K \setminus \{b\}$ (see Paragraph~\ref{ssec:icg}), obviously extended to $K$. 
Then, for all $x, y \in K$, $y \leqslant_b x$ if and only if $x \leqslant b \oplus y$ and $(x \geqslant b$ or $x \geqslant y$). Since the tree is semitopological, the subsets $\uparrow\!\! x$ and $\downarrow\!\! x$ are closed by \cite[Proposition~VI-1.13(ii)]{Gierz03}. Also, the map $y \mapsto b \oplus y$ is continuous, so $\leqslant_b$-principal ideals $\downarrow_b\!\! x = \{ y \in K : y \leqslant_b x \}$ are closed in $K$. 

Now let $x \in K$. If $x$ is minimal in $K$, then $x \in \ex K$. Otherwise, applying Wallace's lemma, there exists some minimal element $b$ of $K$ such that $b < x$ (in particular, $b \in \ex K$). Using Wallace's lemma once more, we find an element $y \in \downarrow_b\!\! x \cap \uparrow\!\! b$, minimal with respect to the partial order $\leqslant_b$. If we show that $y \in \ex K$, we shall have proved that $x \in [b, b\oplus y] \subset \co(\{b, y\}) \subset \co(\ex K)$. 
So write $y \leqslant \bigoplus F$ for some nonempty finite subset $F$ of $K$, and let us see why $y \in F$. In the ambiant tree, $\uparrow b$ is a chain, hence the supremum of $\{ b \oplus f : f \in F\}$ is actually a maximum, i.e.\ there is some $f_0 \in F$ such that $b \oplus f_0 = \bigoplus (b \oplus F) = b \oplus \bigoplus F$. This implies that $b \oplus f_0 \geqslant y \geqslant b$, so that $f_0 \leqslant_b y$. We obtain $y = f_0 \in F$ by minimality of $y$. We conclude that $y$ is max-coprime in $K$, i.e.\ $y \in \ex K$. 
\end{proof}

\subsection{The order-algebraic convexity of a lattice}

Similarly to the above example, the \textit{order-algebraic convexity} of a lattice comprises its order-convex sublattices. The corresponding set of extreme points of a convex subset is the union of its max-coprime and its min-prime (defined dually) elements. Here the condition to get a convex geometry is the same as for ideal convexity. 

\begin{proposition}
A lattice with the order-algebraic convexity is a convex geometry if and only if it is a chain. 
In this case, when endowed with a compatible topology, it satisfies the Krein--Milman property. 
\end{proposition}

\begin{proof}
Following the lines of the proof of Proposition~\ref{prop:fc}, if $y \not\leqslant x$, there is some copoint $C$ at $x$ containing the convex subset $\{ y \}$. The subset $C \cup \{ x \}$ must be convex if the convexity is a convex geometry, so $y \wedge x \in C \cup \{ x \}$ and $y \oplus x \in C \cup \{ x \}$. If both $y \wedge x$ and $y \oplus x$ are in $C$, then $x \in [y \wedge x, y \oplus x] \subset C$ by order-convexity, which contradicts $x \notin C$. Thus, either $y \oplus x \in \{ x \}$ (which is not possible for we assumed $y \not\leqslant x$) or $y \wedge x \in \{ x \}$, i.e.\ $y > x$. 
\end{proof}

\subsection{The algebraic convexity of a lattice}

A final, still challenging example should be evoked.  
On a lattice, one can consider the \textit{algebraic convexity} made up of its sublattices. An abundant literature of topological flavour exists on lattices, and the toolkit of results on locally convex lattices and compact lattices could let one think that the approach adopted for semilattices in Section~\ref{sec:convsemilat} could be reedited without pain. For instance, \cite[Proposition~VII-2.8]{Gierz03} gives a lattice counterpart to the fundamental theorem \ref{thm:fond}.  
Also, Choe \cite{Choe69a, Choe69b} and Stralka \cite{Stralka70} among others studied topological lattices with small lattices, which are nothing but locally convex topological lattices. 

Unfortunately, a deeper examination of this convexity leads to special difficulties.  
Simply consider the fact that extreme points are the \textit{doubly-irreducible} elements (elements that are simultaneously coirreducible for $\leqslant$ and for $\geqslant$), the existence of which is not guaranteed in general, even in finite distributive lattices (look at the power set, ordered by inclusion, of a set with cardinality $> 2$ for instance, see Figure~\ref{fig:Schema4Chapitre4}). On that subject, see also \cite{Poncet13c}. 

\begin{figure}
	\centering
		\includegraphics[width=0.2\textwidth]{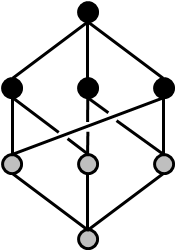}
	\caption{Hasse diagram of the power set of $\{1, 2, 3\}$. The gray (resp.\ black) points are the coirreducible elements with respect to inclusion (resp.\ reverse inclusion). This poset has no doubly-irreducible elements. }
	\label{fig:Schema4Chapitre4}
\end{figure}

The work of Ern\'e \cite{Erne91}, after that of Monjardet and Wille \cite{Monjardet89}, although difficult to interpret, gives some hope in this direction (see also the paper by Berman and Bordalo \cite{Berman98}). 
Rephrasing \cite[Theorem~4.14]{Erne91} for the finite case, one has: 

\begin{proposition}[Monjardet--Wille--Ern\'e]\label{prop:kmlat}
In a finite distributive lattice, the following conditions are equivalent: 
\begin{itemize}
	\item $P$ is principally separated,  
	\item the normal completion of $P$ is a distributive lattice, 
	\item the lattice is generated by its doubly-irreducible elements, 
	\item each coprime is a meet of doubly-irreducible elements, 
	\item for all $p \in P, q \in Q$ with $p \leqslant q$, there exists $r \in P \cap Q : p \leqslant r \leqslant q$, 
\end{itemize}
where $P$ (resp.\ $Q$) denotes the subset of coprime (resp.\ prime) elements. 
\end{proposition}

The normal completion refers to the smallest complete lattice in which a poset embeds (also called Dedekind--MacNeille completion, or completion by cuts). Principal separation in a poset is the assertion that, for all $x \not\leqslant y$, there are some $p\leqslant x$, $q \geqslant y$ such that $p \not\leqslant q$ and $\uparrow\!p \cup \downarrow\!\! q$ is the whole poset; for complete lattices, this is equivalent to complete distributivity.  

A distributive lattice with the algebraic convexity is then a convex geometry if and only if \textit{every} finite sublattice satisfies the conditions of Proposition~\ref{prop:kmlat} (because a polytope is here necessarily finite).

\begin{problem}
Does every compact locally convex distributive lattice (i.e.\ every completely distributive lattice) satisfy the Krein--Milman property as soon as it is a convex geometry? 
\end{problem}

\section{Conclusion and perspectives}

In a future work, we shall consider the natural (algebraic) convexity on \textit{idempotent modules}. 
We shall also aim at relaxing the Hausdorff hypothesis after the work of Goubault-Larrecq \cite{Goubault08}, who proved a Krein--Milman type theorem for non-Hausdorff \textit{cones} (in the sense of Kei\-mel \cite{Keimel08}). 

\begin{acknowledgements}
I am very grateful to Marianne Akian for her crucial help in the proof of Milman's converse for semilattices, and to St\'ephane Gaubert who pointed out to me the direction of proof of the Krein--Milman theorem. 
I also gratefully thank Prof.\ Jimmie D.\ Lawson for his very motivating suggestions on the non-Hausdorff setting, which may be used in a future work. 
\end{acknowledgements}

\bibliographystyle{plain}

\begin{thebibliography}{99}

\bibitem{Aliprantis06}
Charalambos~D. Aliprantis and Kim~C. Border.
\newblock {\em Infinite dimensional analysis}.
\newblock Springer, Berlin, third edition, 2006.
\newblock A hitchhiker's guide.

\bibitem{Baker69}
Kirby~A. Baker.
\newblock A {K}rein--{M}ilman theorem for partially ordered sets.
\newblock {\em Amer. Math. Monthly}, 76:282--283, 1969.

\bibitem{Bauer60}
Heinz Bauer.
\newblock Minimalstellen von {F}unktionen und {E}xtremalpunkte. {II}.
\newblock {\em Arch. Math.}, 11:200--205, 1960.

\bibitem{Berman98}
Joel Berman and Gabriela Bordalo.
\newblock Finite distributive lattices and doubly irreducible elements.
\newblock {\em Discrete Math.}, 178(1-3):237--243, 1998.

\bibitem{Binczak07}
Grzegorz Bi{\'n}czak, Anna~B. Romanowska, and Jonathan D.~H. Smith.
\newblock Poset extensions, convex sets, and semilattice presentations.
\newblock {\em Discrete Math.}, 307(1):1--11, 2007.

\bibitem{Birkhoff48}
Garrett Birkhoff.
\newblock {\em Lattice theory}.
\newblock American Mathematical Society Colloquium Publications, vol. 25,
  revised edition. American Mathematical Society, New York, NY, 1948.

\bibitem{Birkhoff85}
Garrett Birkhoff and Mary~K. Bennett.
\newblock The convexity lattice of a poset.
\newblock {\em Order}, 2(3):223--242, 1985.

\bibitem{Blyth05}
Thomas~S. Blyth.
\newblock {\em Lattices and ordered algebraic structures}.
\newblock Universitext. Springer-Verlag London Ltd., London, 2005.

\bibitem{Bourbaki81}
Nicolas Bourbaki.
\newblock {\em Espaces vectoriels topologiques. {C}hapitres 1 \`a 5}.
\newblock Masson, Paris, new edition, 1981.
\newblock {\'E}l{\'e}ments de math{\'e}matique.

\bibitem{Briec05}
Walter Briec, Charles~D. Horvath, and Aleksandr~M. Rubinov.
\newblock Separation in {$\mathbb B$}-convexity.
\newblock {\em Pac. J. Optim.}, 1(1):13--30, 2005.

\bibitem{Butkovic07}
Peter Butkovi{\v{c}}, Hans Schneider, and Sergei~N. Sergeev.
\newblock Generators, extremals and bases of max cones.
\newblock {\em Linear Algebra Appl.}, 421(2-3):394--406, 2007.

\bibitem{Choe69a}
Tae~Ho Choe.
\newblock Intrinsic topologies in a topological lattice.
\newblock {\em Pacific J. Math.}, 28:49--52, 1969.

\bibitem{Choe69b}
Tae~Ho Choe.
\newblock On compact topological lattices of finite dimension.
\newblock {\em Trans. Amer. Math. Soc.}, 140:223--237, 1969.

\bibitem{Cohen04}
Guy Cohen, St\'ephane Gaubert, and Jean-Pierre Quadrat.
\newblock Duality and separation theorems in idempotent semimodules.
\newblock {\em Linear Algebra Appl.}, 379:395--422, 2004.
\newblock Tenth Conference of the International Linear Algebra Society.

\bibitem{Develin04}
Mike Develin and Bernd Sturmfels.
\newblock Tropical convexity.
\newblock {\em Doc. Math.}, 9:1--27 (electronic), 2004.

\bibitem{Edalat95}
Abbas Edalat.
\newblock Dynamical systems, measures, and fractals via domain theory.
\newblock {\em Inform. and Comput.}, 120(1):32--48, 1995.

\bibitem{Edelman80}
Paul~H. Edelman.
\newblock Meet-distributive lattices and the anti-exchange closure.
\newblock {\em Algebra Universalis}, 10(3):290--299, 1980.

\bibitem{Edelman85}
Paul~H. Edelman and Robert~E. Jamison.
\newblock The theory of convex geometries.
\newblock {\em Geom. Dedicata}, 19(3):247--270, 1985.

\bibitem{Ellis52}
J.~W. Ellis.
\newblock A general set-separation theorem.
\newblock {\em Duke Math. J.}, 19:417--421, 1952.

\bibitem{Erne91}
Marcel Ern{\'e}.
\newblock Bigeneration in complete lattices and principal separation in ordered
  sets.
\newblock {\em Order}, 8(2):197--221, 1991.

\bibitem{Erne06}
Marcel Ern{\'e}, Mai Gehrke, and Al\v{e}s Pultr.
\newblock Complete congruences on topologies and down-set lattices.
\newblock {\em Appl. Categ. Structures}, 15(1-2):163--184, 2007.

\bibitem{Fan63}
Ky~Fan.
\newblock On the {K}rein--{M}ilman theorem.
\newblock In {\em Proc. {S}ympos. {P}ure {M}ath., {V}ol. {VII}}, pages
  211--219. Amer. Math. Soc., Providence, RI, 1963.

\bibitem{Franklin62}
Stanley~P. Franklin.
\newblock Some results on order-convexity.
\newblock {\em Amer. Math. Monthly}, 69(5):357--359, 1962.

\bibitem{Gaubert07}
St\'{e}phane Gaubert and Ricardo~D. Katz.
\newblock The {M}inkowski theorem for max-plus convex sets.
\newblock {\em Linear Algebra Appl.}, 421(2-3):356--369, 2007.

\bibitem{Gierz03}
Gerhard Gierz, Karl~Heinrich Hofmann, Klaus Keimel, Jimmie~D. Lawson, Michael~W. Mislove, and Dana~S. Scott.
\newblock {\em Continuous lattices and domains}, volume~93 of {\em Encyclopedia
  of Mathematics and its Applications}.
\newblock Cambridge University Press, Cambridge, 2003.

\bibitem{Goubault08}
Jean Goubault-Larrecq.
\newblock A cone theoretic {K}rein--{M}ilman theorem.
\newblock Rapport de recherche LSV-08-18, ENS Cachan, France, 2008.

\bibitem{Helbig88}
Siegfried Helbig.
\newblock On {C}arath\'eodory's and {K}re\u\i n--{M}ilman's theorems in fully
  ordered groups.
\newblock {\em Comment. Math. Univ. Carolin.}, 29(1):157--167, 1988.

\bibitem{HofmannLawson76}
Karl~Heinrich Hofmann and Jimmie~D. Lawson.
\newblock Irreducibility and generation in continuous lattices.
\newblock {\em Semigroup Forum}, 13(4):307--353, 1976/77.

\bibitem{Hofmann76}
Karl~Heinrich Hofmann and Albert Stralka.
\newblock The algebraic theory of compact {L}awson semilattices. {A}pplications
  of {G}alois connections to compact semilattices.
\newblock {\em Dissertationes Math. (Rozprawy Mat.)}, 137:58, 1976.

\bibitem{Horvath96}
Charles~D. Horvath and Juan~Vicente Llinares~Ciscar.
\newblock Maximal elements and fixed points for binary relations on topological
  ordered spaces.
\newblock {\em J. Math. Econom.}, 25(3):291--306, 1996.

\bibitem{Jamison74}
Robert~E. Jamison.
\newblock {\em A general theory of convexity}.
\newblock PhD thesis, University of Washington, Seattle, USA, 1974.

\bibitem{Jamison77}
Robert~E. Jamison.
\newblock Some intersection and generation properties of convex sets.
\newblock {\em Compositio Math.}, 35(2):147--161, 1977.

\bibitem{Jamison80}
Robert~E. Jamison.
\newblock Copoints in antimatroids.
\newblock In {\em Proceedings of the {E}leventh {S}outheastern {C}onference on
  {C}ombinatorics, {G}raph {T}heory and {C}omputing ({F}lorida {A}tlantic
  {U}niv., {B}oca {R}aton, {F}la., 1980), {V}ol. {II}}, volume~29, pages
  535--544, 1980.

\bibitem{Jamison79}
Robert~E. Jamison-Waldner.
\newblock A convexity characterization of ordered sets.
\newblock In {\em Proceedings of the {T}enth {S}outheastern {C}onference on
  {C}ombinatorics, {G}raph {T}heory and {C}omputing ({F}lorida {A}tlantic
  {U}niv., {B}oca {R}aton, {F}la., 1979)}, Congress. Numer., XXIII--XXIV, pages
  529--540, Winnipeg, Man., 1979. Utilitas Math.

\bibitem{Jamison81}
Robert~E. Jamison-Waldner.
\newblock Partition numbers for trees and ordered sets.
\newblock {\em Pacific J. Math.}, 96(1):115--140, 1981.

\bibitem{Jamison82}
Robert~E. Jamison-Waldner.
\newblock A perspective on abstract convexity: classifying alignments by
  varieties.
\newblock In {\em Convexity and related combinatorial geometry ({N}orman,
  {O}kla., 1980)}, volume~76 of {\em Lecture Notes in Pure and Appl. Math.},
  pages 113--150. Dekker, New York, 1982.

\bibitem{Kakutani37}
Shizuo Kakutani.
\newblock Ein {B}eweis des {S}atzes von {M}. {E}idelheit \"uber konvexe
  {M}engen.
\newblock {\em Proc. Imp. Acad.}, 13(4):93--94, 1937.

\bibitem{Keimel08}
Klaus Keimel.
\newblock Topological cones: functional analysis in a {${\rm T}\sb 0$}-setting.
\newblock {\em Semigroup Forum}, 77(1):109--142, 2008.

\bibitem{Klee57}
Victor~L. Klee, Jr.
\newblock Extremal structure of convex sets.
\newblock {\em Arch. Math. (Basel)}, 8:234--240, 1957.

\bibitem{Krein40}
Mark Krein and David Milman.
\newblock On extreme points of regular convex sets.
\newblock {\em Studia Math.}, 9:133--138, 1940.

\bibitem{Lassak86}
Marek Lassak.
\newblock A general notion of extreme subset.
\newblock {\em Compositio Math.}, 57(1):61--72, 1986.

\bibitem{Lawson67}
Jimmie~D. Lawson.
\newblock {\em Vietoris mappings and embeddings of topological semilattices}.
\newblock PhD thesis, University of Tennessee, USA, 1967.

\bibitem{Lawson69}
Jimmie~D. Lawson.
\newblock Topological semilattices with small semilattices.
\newblock {\em J. London Math. Soc. (2)}, 1:719--724, 1969.

\bibitem{Lawson71}
Jimmie~D. Lawson.
\newblock The relation of breadth and codimension in topological semilattices.
  {II}.
\newblock {\em Duke Math. J.}, 38:555--559, 1971.

\bibitem{Lawson73}
Jimmie~D. Lawson.
\newblock Intrinsic topologies in topological lattices and semilattices.
\newblock {\em Pacific J. Math.}, 44:593--602, 1973.

\bibitem{Lawson74}
Jimmie~D. Lawson.
\newblock Joint continuity in semitopological semigroups.
\newblock {\em Illinois J. Math.}, 18:275--285, 1974.

\bibitem{Lawson85}
Jimmie~D. Lawson, Michael Mislove, and Hilary~A. Priestley.
\newblock Infinite antichains in semilattices.
\newblock {\em Order}, 2(3):275--290, 1985.

\bibitem{Lea76}
James~W. Lea, Jr.
\newblock Continuous lattices and compact {L}awson semilattices.
\newblock {\em Semigroup Forum}, 13(4):387--388, 1976/77.

\bibitem{Liukkonen83}
John~R. Liukkonen and Michael~W. Mislove.
\newblock Measure algebras of locally compact semilattices.
\newblock In {\em Recent developments in the algebraic, analytical, and
  topological theory of semigroups ({O}berwolfach, 1981)}, volume 998 of {\em
  Lecture Notes in Math.}, pages 202--214. Springer, Berlin, 1983.

\bibitem{Martinez72}
Jorge Martinez.
\newblock Unique factorization in partially ordered sets.
\newblock {\em Proc. Amer. Math. Soc.}, 33:213--220, 1972.

\bibitem{Milman47}
David Milman.
\newblock Characteristics of extremal points of regularly convex sets.
\newblock {\em Doklady Akad. Nauk SSSR (N.S.)}, 57:119--122, 1947.

\bibitem{Monjardet85}
Bernard Monjardet.
\newblock A use for frequently rediscovering a concept.
\newblock {\em Order}, 1(4):415--417, 1985.

\bibitem{Monjardet89}
Bernard Monjardet and Rudolf Wille.
\newblock On finite lattices generated by their doubly irreducible elements.
\newblock In {\em Proceedings of the {O}berwolfach {M}eeting ``{K}ombinatorik''
  (1986)}, volume~73, pages 163--164, 1989.

\bibitem{Nachbin65}
Leopoldo Nachbin.
\newblock {\em Topology and order}.
\newblock Translated from the Portuguese by Lulu Bechtolsheim. Van Nostrand
  Mathematical Studies, No. 4. D. Van Nostrand Co., Inc., Princeton,
  N.J.-Toronto, Ont.-London, 1965.

\bibitem{Park09}
Sehie Park.
\newblock Comments on abstract convexity structures on topological spaces.
\newblock {\em Nonlinear Anal.}, 72(2):549--554, 2010.

\bibitem{Poncet11}
Paul Poncet.
\newblock {\em Infinite-dimensional idempotent analysis: the role of continuous
  posets}.
\newblock PhD thesis, \'Ecole Polytechnique, Palaiseau, France, 2011.

\bibitem{Poncet13c}
Paul Poncet.
\newblock Pruning a poset with veins.
\newblock http://arxiv.org/abs/1301.0759, 2013.

\bibitem{Singer97}
Ivan Singer.
\newblock {\em Abstract convex analysis}.
\newblock Canadian Mathematical Society Series of Monographs and Advanced
  Texts. John Wiley \& Sons Inc., New York, 1997.
\newblock With a foreword by Aleksandr M. Rubinov, A Wiley-Interscience
  Publication.

\bibitem{Stralka70}
Albert~R. Stralka.
\newblock Locally convex topological lattices.
\newblock {\em Trans. Amer. Math. Soc.}, 151:629--640, 1970.

\bibitem{vanDeVel85}
Marcel L.~J. van~de Vel.
\newblock Lattices and semilattices: a convex point of view.
\newblock In {\em Continuous lattices and their applications ({B}remen, 1982)},
  volume 101 of {\em Lecture Notes in Pure and Appl. Math.}, pages 279--302.
  Dekker, New York, 1985.

\bibitem{vanDeVel93b}
Marcel L.~J. van~de Vel.
\newblock A selection theorem for topological convex structures.
\newblock {\em Trans. Amer. Math. Soc.}, 336(2):463--496, 1993.

\bibitem{vanDeVel93}
Marcel L.~J. van~de Vel.
\newblock {\em Theory of convex structures}, volume~50 of {\em North-Holland
  Mathematical Library}.
\newblock North-Holland Publishing Co., Amsterdam, 1993.

\bibitem{NguyenTheVinh05}
Nguyen~The Vinh.
\newblock Matching theorems, fixed point theorems and minimax inequalities in
  topological ordered spaces.
\newblock {\em Acta Math. Vietnam.}, 30(3):211--224, 2005.

\bibitem{Wallace45}
Alexander~D. Wallace.
\newblock A fixed-point theorem.
\newblock {\em Bull. Amer. Math. Soc.}, 51:413--416, 1945.

\bibitem{Wieczorek89}
Andrzej Wieczorek.
\newblock Spot functions and peripherals: {K}rein--{M}ilman type theorems in an
  abstract setting.
\newblock {\em J. Math. Anal. Appl.}, 138(2):293--310, 1989.

\bibitem{Wirth74}
Andrew Wirth.
\newblock Some {K}rein--{M}ilman theorems for order-convexity.
\newblock {\em J. Austral. Math. Soc.}, 18:257--261, 1974.

\end{thebibliography}

\def\cprime{$'$} \def\cprime{$'$} \def\cprime{$'$} \def\cprime{$'$}
  \def\ocirc#1{\ifmmode\setbox0=\hbox{$#1$}\dimen0=\ht0 \advance\dimen0
  by1pt\rlap{\hbox to\wd0{\hss\raise\dimen0
  \hbox{\hskip.2em$\scriptscriptstyle\circ$}\hss}}#1\else {\accent"17 #1}\fi}
  \def\ocirc#1{\ifmmode\setbox0=\hbox{$#1$}\dimen0=\ht0 \advance\dimen0
  by1pt\rlap{\hbox to\wd0{\hss\raise\dimen0
  \hbox{\hskip.2em$\scriptscriptstyle\circ$}\hss}}#1\else {\accent"17 #1}\fi}

\appendix

\section{Some properties of convexities on ordered structures}

\subsection{Arity}

If the convex sets of a convexity are exactly the subsets $C$ such that $\co(F) \subset C$ for all $F \subset C$ with cardinality $\leqslant n$, then the convexity is of \textit{arity} $\leqslant n$. 
All the convexities considered in this paper are of arity $\leqslant 2$. Table~\ref{tab:ar} summarizes special cases.

\begin{table}[ht]
	\centering
	\begin{tabular}{|l|l||c|c|c|c|c|c|c|c|c|c|c|}
	\hline
	  Structure &   Convexity &         Arity &    Arity $= 1$? \\
	\hline\hline
	poset       & upper       &             1 & yes             \\
	\hline
	poset       & order       & $\leqslant 2$ & iff depth $= 2$ \\
	\hline
	semilattice & algebraic   & $\leqslant 2$ &       iff chain \\
	\hline
	semilattice & ideal       & $\leqslant 2$ &       iff chain \\
	\hline
	semilattice & order-alg.\ & $\leqslant 2$ &       iff chain \\
	\hline
	lattice     & order-alg.\ & $\leqslant 2$ &       iff chain \\
	\hline
	lattice     & algebraic   & $\leqslant 2$ &       iff chain \\
	\hline
	\end{tabular}  
	\caption{Arity. }
	\label{tab:ar}
\end{table}

\subsection{Separation axioms}

Convexities are classically classified according to five basic separation axioms, mimicking the usual conditions $T_0, \ldots, T_4$ in topology: 
\begin{enumerate}
	\item[$S_0$.] for each pair of distinct points, there exists a convex set containing one point but not the other, 
	\item[$S_1$.] all singletons are convex,  
	\item[$S_2$.] two distinct points extend to complementary halfspaces, 
	\item[$S_3$.] each convex subset is an intersection of halfspaces, 
	\item[$S_4$.] two disjoint convex subsets extend to complementary halfspaces, 
\end{enumerate}
where a \textit{halfspace} is a convex subset with a convex complement. 

The $S_4$ separation axiom is also called the Kakutani separation property, since Kakutani \cite{Kakutani37} proved its validity in real vector spaces with their usual (Euclidian) convexity. 
Ellis \cite{Ellis52} gave an abstract version of Kakutani's result, that we recall below. Briec et al.\ \cite[Theorem~2.1]{Briec05} gave a self-contained proof in the framework of finite-dimensional tropical geometry, restating arguments due to van de Vel.  

\begin{proposition}
On a poset, the upper convexity (resp.\ the lower convexity) is $S_0$ (but not $S_1$, unless the partial order is trivial), 
the order convexity is $S_3$, and the order convexity on a chain is $S_4$. 
\end{proposition}

\begin{proof}
We prove that the order convexity on a poset is $S_3$. 
Let $C$ be an order-convex subset and $x \notin C$. If $C \cap \downarrow\!\! x = \emptyset$, then $\downarrow\!\! x$ is a halfspace separating $C$ and $x$. The case $C \cap \uparrow\!\! x = \emptyset$ is similar. Otherwise, there exists some $y \in C \cap \downarrow\!\! x$ and $z \in C \cap \uparrow\!\! x$, hence $y \leqslant x \leqslant z$. Since $C$ is order-convex, we have $x \in C$, a contradiction. 
\end{proof}

\begin{proposition}
On a semilattice, the algebraic and the order-algebraic convexities are $S_4$. 
\end{proposition}

\begin{proof}
The case of the order-algebraic convexity is treated by van de Vel \cite[Proposition~I-3.12.2]{vanDeVel93}. 
The algebraic convexity is of arity $2$ and clearly satisfies the Pasch property (see the definition in \cite[Paragraph~I-4.9]{vanDeVel93}), hence is $S_4$ by \cite[Theorem~4.12]{vanDeVel93}. 
\end{proof}

\begin{proposition}\label{prop:s2}
On a lattice that is a distributive continuous semilattice (or dually), in particular on a completely distributive lattice, the algebraic convexity is $S_2$. 
\end{proposition}

\begin{proof}
By \cite[Corollary~I-3.13]{Gierz03}, if $L$ is a distributive continuous semilattice, then its subset of coprime elements is order-generating. Hence, if $x \not\leqslant y$, one can find some coprime element $p$ with $p \leqslant x$ and $p \not\leqslant y$. This implies that $\uparrow\!\! p$, which is a halfspace with respect to the algebraic convexity on the lattice $L$, separates $x$ and $y$. 
\end{proof}

\begin{table}[ht]
	\centering
	\begin{tabular}{|l|l||c|c|c|c|c|c|c|c|c|c|c|}
	\hline
	  Structure &   Convexity &         $S_1$ &            $S_2$ \\
	\hline\hline
	poset       & upper        & iff antichain &    iff antichain \\
	\hline
	poset       & order       &             yes &                yes \\
	\hline
	semilattice & algebraic   &             yes &                yes \\
	\hline
	semilattice & ideal       & iff antichain &    iff antichain \\
	\hline
	semilattice & order-alg.\ &             yes &                yes \\
	\hline
	lattice     & order-alg.\ &             yes & iff distributive \\
	\hline
	lattice     & algebraic   &             yes &if distrib.\ continuous \\
	\hline
	\end{tabular}  
	\caption{$S_1$ and $S_2$ axioms. All these structures satisfy the $S_0$ axiom. For lattices with the order-algebraic convexity, see \cite[Proposition~I-3.12.3]{vanDeVel93}. }
	\label{tab:sep12}
\end{table}

\begin{table}[ht]
	\centering
	\begin{tabular}{|l|l||c|c|c|c|c|c|c|c|c|c|c|}
	\hline
	  Structure &   Convexity &            $S_3$ &            $S_4$ \\
	\hline\hline
	poset       & upper       &    iff antichain &    iff antichain \\
	\hline
	poset       & order       &                yes &         if chain \\
	\hline
	semilattice & algebraic   &                yes &                yes \\
	\hline
	semilattice & ideal       &    iff antichain &    iff antichain \\
	\hline
	semilattice & order-alg.\ &                yes &                yes \\
	\hline
	lattice     & order-alg.\ & iff distributive & iff distributive \\
	\hline
	lattice     & algebraic   &                ? &                ? \\
	\hline
	\end{tabular}  
	\caption{$S_3$ and $S_4$ axioms. For lattices with the order-algebraic convexity, see \cite[Proposition~I-3.12.3]{vanDeVel93}. }
	\label{tab:sep34}
\end{table}

\subsection{The convex geometry property}

Table~\ref{tab:cg} recalls several results of Section~\ref{sec:cg}. 

\begin{table}[ht]
	\centering
	\begin{tabular}{|l|l||c|c|c|c|c|c|c|c|c|c|c|}
	\hline
	  Structure &   Convexity & Convex geometry & Extreme points \\
	\hline\hline
	poset       & upper       &               yes & minimal \\
	\hline
	poset       & order       &               yes & minimal or maximal \\
	\hline
	semilattice & algebraic   &               yes & coirreducible \\
	\hline
	semilattice & ideal       &       iff chain & max-coprime \\
	\hline
	semilattice & order-alg.\ &        iff tree & minimal or max-coprime \\
	\hline
	lattice     & order-alg.\ &       iff chain & min-prime or max-coprime \\
	\hline
	lattice     & algebraic   &               ? & doubly-irreducible \\
	\hline
	\end{tabular}  
	\caption{Convex geometry property and extreme points. }
	\label{tab:cg}
\end{table}

\end{document}